\documentclass{amsart}

\usepackage{amsfonts, amsmath, amscd}
\usepackage[psamsfonts]{amssymb}
\usepackage{amssymb}
\usepackage[usenames]{color}
\usepackage{hyperref}
\usepackage[all]{xy}
\numberwithin{equation}{section}

\newtheorem{theorem}{Theorem}[section]
\newtheorem{corollary}[theorem]{Corollary}
\newtheorem{proposition}[theorem]{Proposition}
\newtheorem{lemma}[theorem]{Lemma}
\theoremstyle{definition}
\newtheorem{example}[theorem]{Example}
\newtheorem{definition}[theorem]{Definition}

\newtheorem{problem}[theorem]{Problem}
\newtheorem{remark}[theorem]{Remark}

\theoremstyle{plain}

\newcommand{\w}{\omega}
\newcommand{\K}{\mathcal K}

\newcommand{\Ra}{\Rightarrow}
\newcommand{\IZ}{\mathbb Z}
\newcommand{\IR}{\mathbb R}
\newcommand{\IN}{\mathbb N}

\newcommand{\V}{\mathcal V}

\newcommand{\U}{\mathcal U}
\newcommand{\F}{\mathcal F}
\newcommand{\E}{\mathcal E}
\newcommand{\e}{\varepsilon}
\newcommand{\Lc}{\mathsf{L}}
\newcommand{\Lin}{\mathsf{V}}
\newcommand{\FG}{\mathsf{F}}
\newcommand{\AG}{\mathsf{A}}
\newcommand{\BG}{\mathsf{B}}
\newcommand{\FO}{\mathsf{O}}

\newcommand{\St}{\mathcal{S}t}
\newcommand{\Tau}{\mathcal{T}}
\newcommand{\conv}{\mathrm{conv}}
\newcommand{\cov}{\mathsf{cov}}
\newcommand{\add}{\mathsf{add}}
\newcommand{\cof}{\mathsf{cof}}
\newcommand{\cf}{\mathsf{cof}}
\newcommand{\C}{\mathcal C}
\newcommand{\supp}{\mathrm{supp}}
\begin{document}

\title[$\w^\w$-Bases in Free Objects of Topological Algebra]{$\w^\w$-Dominated Function Spaces and \\ $\w^\w$-Bases in Free Objects of Topological Algebra}
%{$\w^\w$-bases in free (Abelian) topological groups and free (locally convex) linear topological spaces}
\author{Taras Banakh and Arkady Leiderman}
\address{T. Banakh: Department of Mathematics, Ivan Franko National University of Lviv (Ukraine) and
\newline Instytut Matematyki, Jan Kochanowski University in Kielce (Poland)}
\email{t.o.banakh@gmail.com}
\address{A. Leiderman: Department of Mathematics, Ben-Gurion University of the Negev, Beer Sheva, P.O.B. 653, Israel}
\email{arkady@math.bgu.ac.il}

\date{\today}

\begin{abstract}
A topological space $X$ is defined to have an {\em $\w^\w$-base} if at each point $x\in X$ the space $X$ has a neighborhood base $(U_\alpha[x])_{\alpha\in\w^\w}$ such that $U_\beta[x]\subset U_\alpha[x]$ for  all $\alpha\le\beta$ in $\w^\w$.
\newline
For a Tychonoff space $X$ consider the following conditions:
\begin{enumerate}
\item[$(\AG)$] the free Abelian topological group $\AG(X)$ of $X$ has an $\w^\w$-base;
\item[$(\BG)$] the free Boolean topological group $\BG(X)$ of $X$ has an $\w^\w$-base;
\item[$(\FG)$] the free topological group $\FG(X)$ of $X$ has an $\w^\w$-base;
\item[$(\Lc)$] the free locally convex space $\Lc(X)$ of $X$ has an $\w^\w$-base;
\item[$(\Lin)$] the free topological vector space $\Lin(X)$ of $X$ has an $\w^\w$-base;
\item[$(\U)$] the universal uniformity $\U_X$ of $X$ has a base
 $(U_\alpha)_{\alpha\in\w^\w}$ such that $U_\beta\subset U_\alpha$ for all $\alpha\le\beta$ in $\w^\w$;
\item[$(C)$] the function space $C(X)$ is $\w^\w$-dominated;

\item[$(\sigma)$] $X$ is $\sigma$-compact;
\item[$(\sigma')$] the set $X'$ of non-isolated points in $X$ is $\sigma$-compact.
\item[$(s)$] the space $X$ is separable;
\item[$(S)$] $X$ is separable or $\cov^\sharp(X)\le\add(X)$;
\item[$(D)$] $X$ is discrete.
\end{enumerate}
Then
$(\Lc)\Leftrightarrow (\Lin)\Leftrightarrow(\U{\wedge}C)\Leftrightarrow(\U{\wedge}\sigma)\Leftrightarrow(\U{\wedge}s)\Ra(\U{\wedge}S)\Ra(\FG)\Ra(\AG)\Leftrightarrow(\BG)\Leftrightarrow(\U)$ and moreover  $(\U{\wedge}S)\Leftrightarrow (\FG)$ under the set-theoretic assumption $\mathfrak e^\sharp=\w_1$ (which is weaker than $\mathfrak b=\mathfrak d$).
\newline If $X$ is not a $P$-space, then $(\Lc)\Leftrightarrow (\Lin)\Leftrightarrow(\U{\wedge}C)\Leftrightarrow(\U{\wedge}\sigma)\Leftrightarrow(\U{\wedge}s)\Leftrightarrow (\FG)\Ra(\AG)\Leftrightarrow(\BG)\Leftrightarrow(\U)$.\newline
If the space $X$ is metrizable, then $(\Lc)\Leftrightarrow (\Lin)\Leftrightarrow(\sigma)\Ra(D{\vee}\sigma)\Leftrightarrow (\FG)\Ra(\AG)\Leftrightarrow(\BG)\Leftrightarrow(\sigma')$.
\end{abstract}
\maketitle

\tableofcontents

\section{Introduction and Main Results}
It is known that a uniform space is metrizable if and only if its uniformity  has a countable base (see \cite[8.1.21]{Eng}).  Natural compatible uniform structures exist in topological groups, and, in particular, in topological vector spaces. This fact implies
the classical metrization theorem of Birkhoff and Kakutani which states that a topological group is metrizable if and only if it admits a countable base of neighborhoods of the identity, or, in other words, a base indexed by $\omega$.

In this paper we consider (topological) uniform spaces which have (neighborhood) bases indexed by a more complicated directed set $\w^\w$, which consists of all functions from $\w$ to $\w$ and is endowed with the natural partial order (defined by $f\le g$ iff $f(n)\le g(n)$ for all $n\in\w$).
Cardinal invariants of topological spaces, associated with combinatorial properties of the poset $\w^\w$
have been investigated for many years (see \cite{Douwen}, \cite{Vaug}).

A topological space $X$ is defined to have a {\em neighborhood $\w^\w$-base} at a point $x\in X$ if there exists
a neighborhood base $(U_\alpha[x])_{\alpha\in\w^\w}$ at $x$ such that $U_\beta[x]\subset U_\alpha[x]$ for all
elements $\alpha\le\beta$ in $\w^\w$. We shall say that a topological space has an {\em $\w^\w$-base} if it has a neighborhood $\w^\w$-base at each point $x\in X$. Evidently, a topological group has an  $\w^\w$-base if it has a neighborhood $\w^\w$-base at the identity.

We shall say that a uniform space $X$ has an {\em $\w^\w$-base} if its uniformity $\U_X$ has a base of entourages $\{U_\alpha\}_{\alpha\in\w^\w}\subset\U_X$ such that $U_\beta\subset U_\alpha$ for all $\alpha\le\beta$ in $\w^\w$. In this case the base $\{U_\alpha\}_{\alpha\in\w^\w}$ will be called an {\em $\w^\w$-base} of the uniformity $\U_X$ and the uniform space $X$ will be called {\em $\w^\w$-based}.

For a Tychonoff space $X$ by $\U_X$ we denote {\em the universal uniformity} of $X$, i.e., the largest uniformity, compatible with the topology of $X$. It is generated by the base consisting of the entourages $[\rho]_{<1}:=\{(x,y)\in X\times X:\rho(x,y)<1\}$ where $\rho$ runs over all continuous pseudometrics on $X$. A Tychonoff space $X$ will be called {\em universally $\w^\w$-based} if its universal uniformity has an $\w^\w$-base.

%\end{itemize}

Cascales and Orihuela \cite{CO} were the first who considered (universally) $\w^\w$-based spaces.
They proved that compact Hausdorff spaces with this property are metrizable (recall that a compact space carries a unique compatible uniformity). For the first time the concept of an $\w^\w$-base  appeared in \cite{FKLS} as a tool for studying locally convex spaces that belong to the class $\mathfrak{G}$ introduced by Cascales and Orihuela \cite{CO}. A systematic study of locally convex spaces and topological groups with an $\w^\w$-base was started in \cite{GabKakLei_1}, \cite{GabKakLei_2} and continued in \cite{GabKak_1},
 \cite{GabKak_2}, \cite{LRZ}.  In these papers $\w^\w$-bases are called $\mathfrak G$-bases, but we prefer to use the terminology of $\w^\w$-bases, which is more self-suggesting and flexible. Implicitly, $\w^\w$-bases in topological groups appear also in \cite{CFHT}, \cite{GartMor} and \cite{LT}. Topological and uniform spaces with an $\w^\w$-base were systematically studied in the paper \cite{Ban}.
The class of topological spaces with an $\w^\w$-base is closed under natural operations: taking subspaces, countable Tychonoff products, and countable box-products \cite[\S6.1]{Ban}. The class of topological groups with an $\w^\w$-base is closed under taking quotients and  the (Raikov) completions, see \cite[2.7]{GabKakLei_2}.

In this paper we detect topological and uniform spaces $X$ whose
free Abelian topological group $\AG(X)$, free Boolean topological group $\BG(X)$,
free topological group $\FG(X)$,
free locally convex space $\Lc(X)$, and
free topological vector space $\Lin(X)$ have $\w^\w$-bases.

Now we recall the definitions of these free constructions of Topological Algebra and Functional Analysis. All topological groups and topological vector spaces are assumed to be Hausdorff, and therefore Tychonoff. Each topological group $G$ will be endowed with the two-sided uniformity generated by the base consisting of entourages $\{(x,y)\in G:y\in Ux\cap xU\}$ where $U$ runs over neighborhoods of the unit in $G$.

For a uniform space $X$ its {\em free Abelian topological group} is a pair $(\AG_u(X),\delta_X)$ consisting of an Abelian topological group $\AG_u(X)$ and a uniformly continuous map $\delta_X:X\to \AG_u(X)$ such that for every uniformly continuous map $f:X\to Y$ into an Abelian topological group $Y$ there exists a continuous group homomorphism $\bar f:\AG(X)\to Y$ such that $\bar f\circ\delta_X=f$. Deleting the adjective ``Abelian'' from this definition, we get the definition of a {\em free topological group} $(\FG_u(X),\delta_X)$ over a uniform space $X$.
A {\em Boolean group} is a group in which all elements are of order two. Clearly, all Boolean groups are Abelian.
Replacing the adjective ``Abelian'' in the definition of $(\AG_u(X),\delta_X)$  by ``Boolean'',
we get the definition of a {\em free Boolean topological group} $(\BG_u(X),\delta_X)$ over a uniform space $X$.

For a uniform space $X$ its {\em free locally convex space} is a pair $(\Lc_u(X),\delta_X)$ consisting of a locally convex topological vector space $\Lc_u(X)$ and a continuous map $\delta_X:X\to \Lc_u(X)$ such that for every uniformly continuous map $f:X\to Y$ into a locally convex topological vector space $Y$ there exists a continuous linear operator $\bar f:\Lc(X)\to Y$ such that $\bar f\circ\delta_X=f$. Deleting the phrase ``locally convex'' in this definition, we obtain the definition of a {\em free topological vector space} $(\Lin_u(X),\delta_X)$ of $X$. All topological vector spaces considered in this paper are over the field $\IR$ of real numbers.

It is well-known that for every uniform space $X$ its free (Abelian/Boolean) topological group and free (locally convex) linear topological space exists and is unique up to a topological isomorphism.
The canonical map $\delta_X$ is a homeomorphic embedding, and we may say that in fact $X$ algebraically generates the free groups $\AG_u(X)$, $\BG_u(X)$, $\FG_u(X)$ or the free vector spaces $\Lc_u(X)$, $\Lin_u(X)$.

For a Tychonoff space $X$ we put $\AG(X)=\AG_u(X)$, $\BG(X)=\BG_u(X)$, $\FG(X)=\FG_u(X)$, $\Lc(X)=\Lc_u(X)$, $\Lin(X)=\Lin_u(X)$ where $X$ is considered as a uniform space endowed with the universal uniformity of $X$. In this case any continuous map $X\to Y$ to a uniform space $Y$ is uniformly continuous. %This uniformity is generated by the base consisting of the entourages $[\rho]_{<1}:=\{(x,y)\in X\times X:\rho(x,y)<1\}$ where $\rho$ runs over all continuous pseudometrics on $X$.

Note that if a Tychonoff space $X$ is infinite, then $\Lc(X)$ and $\Lin(X)$ are not metrizable; $\AG(X)$, $\BG(X)$ and $\FG(X)$ are metrizable if and only if $X$ is discrete.
A reader is advised to look at two survey papers \cite{Sipa1}, \cite {Sipa2}
 for a detail description of the constructions and basic properties of $\AG(X)$, $\FG(X)$ and $\BG(X)$. Free topological vector spaces $\Lin(X)$ are systematically studied in \cite{GabMor}.
%A preliminary version of our paper which contained results about $\Lc(X)$ and $\Lin(X)$ was announced in \cite{BanLei}.

%A topological space $X$ is called a {\it $k_\omega$-space} if $X$ admits a countable cover $\K$ by compact subsets such that a subset $V\subseteq X$ is open in $X$ if and only if $V\cap K$ is open in $K$ for every $K\in\K$.

It was shown earlier that
\begin{enumerate}%\itemsep=0pt\parskip=0pt
\item[(a)] for a cosmic $k_\w$-space the free objects $\AG(X)$,  $\FG(X)$, $\Lc(X)$ have $\w^\w$-bases \cite{GabKakLei_2}, \cite{GabKak_2}, \cite{GK_L(X)}, \cite{LRZ}.
\item[(b)] for a metrizable $\sigma$-compact space $X$ the free locally convex space $\Lc(X)$ has an $\w^\w$-base \cite{GK_L(X)};\footnote{This result (and its counterpart for the free topological vector spaces) was also proved in the unpublished manuscript \cite{BanLei} whose results are included and strengthened in this paper.}
%\item[(c)] for a countable Ascoli space with an $\w^\w$-base the space $\Lc(X)$ has an $\w^\w$-base \cite{GK_L(X)};
\item[(c)] the free Abelian topological group $\AG(X)$ of a Tychonoff space $X$ has an $\w^\w$-base if and only if the universal uniformity $\U_X$ of $X$ has an $\w^\w$-base \cite{LPT}, \cite{LT}, \cite{GartMor};
\item[(d)] a uniform space $X$ is $\w^\w$-based iff the free Abelian topological group $\AG_u(X)$ has an $\w^\w$-base \cite{LPT};
\item[(e)] a separable uniform space $X$ is $\w^\w$-based iff the free topological group $\FG_u(X)$ has an $\w^\w$-base \cite{LPT}.
%\item[(e)] the universal uniformity $\U_X$ of $X$ has an $\w^\w$-base if $X$ is metrizable and $\sigma'$-compact \cite{LPT};
%\item[(f)] the universal uniformity $\U_X$ of $X$ has an $\w^\w$-base if $X$ has an $\w^\w$-base and the set $X'$ of non-isolated points in $X$ is countable \cite{LPT}.
%\item[(g)] $\Lc(X)$ need not have an $\w^\w$-base even if $\AG(X)$ does \cite{LPT}, \cite{GK_L(X)}.
\end{enumerate}

In this paper we shall characterize uniform spaces whose free topological groups and free (locally convex) topological vector spaces have $\w^\w$-bases. To formulate our principal results, we need to introduce two cardinal characteristics $\cov^\sharp(X)$ and $\add(X)$ of a uniform space $X$.

A uniform space $X$ is called {\em discrete} if its uniformity $\U_X$ contains the diagonal $\Delta_X$ of the square $X\times X$. For a non-discrete uniform space $X$ the (well-defined) cardinal
$$\add(X)=\min\{|\V|:\textstyle{\V\subset\U_X,\;\;\bigcap\V\notin\U_X}\}$$
is called the {\em additivity} of the uniform space $X$. Here $|A|$ stands for the cardinality of a set $A$. For a discrete uniform space $X$ the cardinal $\add(X)$ is not defined. In this case we put $\add(X)=\infty$ and assume that $\infty>\kappa$ for any cardinal $\kappa$.

Elements of the uniformity $\U_X$ are called {\em entourages} on $X$. For an entourage $U\in\U_X$ and a point $x\in X$ the set
$U[x]=\{y\in U:(x,y)\in U\}$ is called the {\em $U$-ball\/} centered at $x$. For a subset $A\subset X$  the set $U[A]=\bigcup_{a\in A}U[a]$ is the {\em $U$-neighborhood\/} of $A$. The cardinal
$$\cov(A;U)=\min\{|B|:B\subset X,\;A\subset U[B]\}$$is called the {\em $U$-covering number} of $A$ and the cardinal $$\cov^\sharp(A;\U_X)=\min\{\kappa:\forall U\in\U_X\;\;\cov(A;U)<\kappa\}$$is the {\em sharp covering number} of $A$ in $X$. It is equal to the smallest cardinal $\kappa$ such that for any entourage $U\in\U_X$ there exists a subset $B\subset X$ of cardinality $|B|<\kappa$ such that $A\subset U[B]$. Consequently, a subset $A$ of a uniform space $X$ is totally bounded (in the standard sense \cite[\S8.3]{Eng}) if and only if $\cov^\sharp(A;\U_X)\le\w$. The cardinal $\cov^\sharp(X;\U_X)$ is denoted by $\cov^\sharp(X)$ and called the {\em sharp covering number} of the uniform space $X$. A uniform space $X$ with sharp covering number $\cov^\sharp(X)\le\w_1$ is called {\em $\w$-narrow}.

For a uniform space $X$ by $C_u(X)$ we denote the space of all uniformly continuous real-valued functions on $X$. It is endowed with the partial order $\le$ defined by $f\le g$ iff $f(x)\le g(x)$ for all $x\in X$. The poset $C_u(X)$ is called {\em $\w^\w$-dominated} if there exists a cofinal subset $\{f_\alpha\}_{\alpha\in\w^\w}\subset C_u(X)$ such that $f_\alpha\le f_\beta$ for all $\alpha\le\beta$ in $\w^\w$. The cofinality of $\{f_\alpha\}_{\alpha\in\w^\w}$ means that for every $f\in C_u(X)$ there exists $\alpha\in\w^\w$ such that $f\le f_\alpha$.

A uniform space $X$ is called {\em universal} if every continuous map $f:X\to M$ to a metric space $M$ is uniformly continuous. This happens if and only if the uniformity of $X$ coincides with the universal uniformity of the topological space $X$. For a universal uniform space $X$ the function space $C_u(X)$ coincides with the space $C(X)$ of all continuous real-valued functions on $X$.

A function $f:X\to Y$ between uniform spaces is called {\em $\w$-continuous} if for any entourage $U\in\U_Y$ there exists a countable family of entourages $\V\subset\U_X$ such that for any point $x\in X$ there exists an entourage $V\in\V$ such that $f(V[x])\subset U[f(x)]$.

A uniform space $X$ is called
\begin{itemize}
\item {\em $\w$-universal} if each $\w$-continuous map $f:X\to Y$ to a separable metric space $Y$ is uniformly continuous;
\item {\em $\IR$-universal} if each $\w$-continuous map $f:X\to\IR$ is uniformly continuous.
\end{itemize}
It is clear that for any uniform space $X$ we have the implications:
$$\mbox{universal $\Ra$ $\w$-universal $\Ra$ $\IR$-universal.}$$

A uniform space $X$ is called {\em $\IR$-complete} if the canonical map $\delta:X\to\IR^{C_u(X)}$ assigning to each $x\in X$ the Dirac measure $\delta_x:C_u(X)\to\IR$, $\delta_x:f\mapsto f(x)$, is a closed topological embedding.

By \cite[3.11.6]{Ban}, the completion of any $\IR$-universal $\w$-narrow uniform space is $\IR$-complete. An example of a complete uniform space which is not $\IR$-complete is any non-compact closed convex bounded subset of a Banach space, see \cite[3.11.7]{Ban}.

A topological space $X$ is called a {\em $P$-space} if each $G_\delta$-set in $X$ is open. For a Tychonoff space $X$ by $\cov^\sharp(X)$ and $\add(X)$ we understand the sharp covering number and the additivity of the uniform space $(X,\U_X)$ where $\U_X$ is the  universal uniformity of $X$.
%For a uniform space $X$ let $\cov_\w^\sharp(X)$ be the smallest cardinal $\kappa$ such that $X$ can be written as the countable union $X=\bigcup_{n\in\w}X_n$ of subsets with $\cov^\sharp(X_n;\U_X)\le\kappa$ for all $n\in\w$. A uniform space $X$ is called {\em $\sigma$-bounded} if $\cov^\sharp_\w(X)\le\w$, i.e.,$X$ is a countable union of totally bounded subsets of $X$. It is clear that each $\sigma$-bounded uniform space $X$ is $\w$-narrow.

One of the main results of this paper is the following characterization.

\begin{theorem}\label{t:mainU} For a uniform space $X$ consider the following conditions:
\begin{enumerate}
\item[$(\AG)$] the free Abelian topological group $\AG_u(X)$ of $X$ has an $\w^\w$-base;
\item[$(\BG)$] the free Boolean topological group $\BG_u(X)$ of $X$ has an $\w^\w$-base;
\item[$(\FG)$] the free topological group $\FG_u(X)$ of $X$ has an $\w^\w$-base;
\item[$(\Lc)$] the free locally convex space $\Lc_u(X)$ of $X$ has an $\w^\w$-base;
\item[$(\Lin)$] the free topological vector space $\Lin_u(X)$ of $X$ has an $\w^\w$-base;
\item[$(\U)$] the uniformity $\U_X$ of $X$ has an $\w^\w$-base;
\item[$(C)$] the poset $C_u(X)$ is $\w^\w$-dominated;
\item[$(\sigma)$] the space $X$ is $\sigma$-compact;
\item[$(s)$] the space $X$ is separable;
\item[$(S)$] $X$ is separable or $\cov^\sharp(X)\le\add(X)$.
\end{enumerate}
Then
\begin{enumerate}
\item[\textup{(1)}]
$(\U{\wedge}C)\Leftrightarrow(\Lc)\Leftarrow (\Lin)\Leftarrow(\U{\wedge}s)\Leftarrow(\U{\wedge}\sigma)\Ra(\U{\wedge}S)\Ra(\FG)\Ra(\AG)\Leftrightarrow (\BG)\Leftrightarrow(\U)$ and moreover, $(\U{\wedge}s)\Leftrightarrow(\Lin)$ under $\w_1<\mathfrak b$.
\item[\textup{(2)}]
If the completion of the uniform space $X$ is $\IR$-complete, then $(\U\wedge C)\Leftrightarrow(\Lc)\Leftrightarrow(\Lin)\Leftrightarrow(\U{\wedge}s)$.
\item[\textup{(3)}]
If the uniform space $X$ is $\IR$-universal, then $(\U\wedge C)\Leftrightarrow(\Lc)\Leftrightarrow(\Lin)\Leftrightarrow(\U{\wedge}s)\Leftrightarrow(\U{\wedge}\sigma)$.
\item[\textup{(4)}]
If the space $X$ is separable, then $(\Lc)\Leftrightarrow (\Lin)\Leftrightarrow(\FG)\Leftrightarrow(\AG)\Leftrightarrow(\BG)\Leftrightarrow(\U)$.
\item[\textup{(5)}]
If the uniform space $X$ is $\w$-universal and $X$ is not a $P$-space, then\newline
$(\U{\wedge}C)\Leftrightarrow(\Lc)\Leftrightarrow (\Lin)\Leftrightarrow(\U{\wedge}s)\Leftrightarrow(\U{\wedge}\sigma)\Leftrightarrow(\FG)$.
\end{enumerate}
\end{theorem}

Applying Theorem~\ref{t:mainU} to Tychonoff spaces endowed with their universal uniformity, we get the following characterization.

\begin{theorem}\label{t:mainT} For a Tychonoff space $X$ consider the following conditions:
\begin{enumerate}
\item[$(\AG)$] the free Abelian topological group $\AG(X)$ of $X$ has an $\w^\w$-base;
\item[$(\BG)$] the free Boolean topological group $\BG(X)$ of $X$ has an $\w^\w$-base;
\item[$(\FG)$] the free topological group $\FG(X)$ of $X$ has an $\w^\w$-base;
\item[$(\Lc)$] the free locally convex space $\Lc(X)$ of $X$ has an $\w^\w$-base;
\item[$(\Lin)$] the free topological vector space $\Lin(X)$ of $X$ has an $\w^\w$-base;
\item[$(\U)$] the universal uniformity $\U_X$ of $X$ has an $\w^\w$-base;
\item[$(C)$] the poset $C(X)$ is $\w^\w$-dominated;
\item[$(s)$] the space $X$ is separable;
\item[$(S)$] $X$ is separable or $\cov^\sharp(X)\le\add(X)$.
\item[$(D)$] $X$ is discrete;
\item[$(\sigma)$] $X$ is $\sigma$-compact;
\item[$(\sigma')$] the set $X'$ of non-isolated points of $X$ is $\sigma$-compact.
\end{enumerate}
Then
\begin{enumerate}
\item[\textup{(1)}]
$(\U{\wedge}C)\Leftrightarrow(\Lc)\Leftrightarrow (\Lin)\Leftrightarrow(\U{\wedge}s)\Leftrightarrow(\U{\wedge}\sigma)\Ra(\U{\wedge}S)\Ra(\FG)\Ra(\AG)\Leftrightarrow(\BG)\Leftrightarrow(\U)$ \newline and moreover  $(\U{\wedge}S)\Leftrightarrow (\FG)$ under the set-theoretic assumption $\mathfrak e^\sharp=\w_1$ (which is weaker than $\mathfrak b=\mathfrak d$).
\item[\textup{(2)}]
If $X$ is not a $P$-space, then $(\U{\wedge}s)\Leftrightarrow(\U{\wedge}\sigma)\Leftrightarrow(\FG)$.
\item[\textup{(3)}]
If  $X$ is separable, then $(\Lc)\Leftrightarrow (\Lin)\Leftrightarrow(\FG)\Leftrightarrow(\AG)\Leftrightarrow(\BG)\Leftrightarrow(\U)$.
\item[\textup{(4)}]
If  $X$ is metrizable, then
$(\Lc)\Leftrightarrow (\Lin)\Leftrightarrow(\sigma)\Ra(D{\vee}\sigma)\Leftrightarrow(\FG)\Ra(\AG)\Leftrightarrow(\BG)\Leftrightarrow(\sigma')$. \end{enumerate}
\end{theorem}

Theorems~\ref{t:mainU} and \ref{t:mainT} will be proved in Sections~\ref{s:pf1} and \ref{s:pf2}, respectively.

\begin{remark} Theorem~\ref{t:mainT} completely resolves the problem of characterization of Tychonoff spaces $X$ whose free objects $\AG(X)$, $\FG(X)$, $\Lc(X)$ have $\mathfrak G$-bases, posed in \cite[Question 4.15]{GabKakLei_2} and then repeated (and partially answered) in \cite[\S3]{GabKak_2}, \cite{GK_L(X)}, \cite{LPT}, \cite{LRZ}. Similar problems are also considered and (partly) resolved in \cite[\S5.2 and \S6.5]{GartMor}.
\end{remark}

\section{Reductions between posets}\label{s:pre}
For handling $\w^\w$-bases in topological or uniform spaces, it is convenient to use the language of reducibility of posets, exploited also in \cite{Ban} and \cite{GartMor}. By a {\em poset} we understand a partially ordered set, i.e., a non-empty set endowed with a partial order.

A function $f:P\to Q$ between two posets is called
\begin{itemize}
\item  {\em monotone} if $f(p)\le f(p')$ for all $p\le p'$ in $P$;
\item {\em cofinal} if for each $q\in Q$ there exists $p\in P$ such that $q\le f(p)$.
\end{itemize}

Given two partially ordered sets $P,Q$ we shall write $P\succcurlyeq Q$ or $Q\preccurlyeq P$ and say that $Q$ {\em reduces} to $P$ if there exists a monotone cofinal map $f:P\to Q$. Also we write $P\cong Q$ if $P\preccurlyeq Q$ and $P\succcurlyeq Q$. This kind of reducibility of posets  is a bit stronger than the Tukey reducibility $\le_T$, which requires the existence of a function $f:P\to Q$ that maps cofinal subsets of $P$ to cofinal subsets of $Q$. However, if the poset  $Q$ is lower complete (which means that each nonempty set in $Q$ has a largest lower bound), then the reducibility $Q\preccurlyeq P$ is equivalent to the Tukey reducibility $Q\le_T P$, see \cite[2.2.1]{Ban}.

We shall also use a related notion of $P$-dominance. Given two posets $P,Q$, we shall say that a subset $D\subset Q$ is {\em $P$-dominated in $Q$} if there exists a monotone map $f:P\to Q$ such that for every $x\in D$ there exists $y\in f(P)$ with $x\le y$. It follows that a poset $Q$ reduces to $P$ if and only if $Q$ is $P$-dominated in $Q$. In this case we shall also say that the poset $Q$ is {\em $P$-dominated}.

Now we present examples of posets that will appear in this paper. Each cardinal $\kappa$ will be considered as a poset endowed with its natural well-order $\le$ (defined by $x\le y$ iff $x=y$ or $x\in y$ for $x,y\in\kappa$). For any cardinals $\kappa,\lambda$ the power $\kappa^\lambda$ is a poset endowed with the partial order $\le$ defined by $f\le g$ iff $f(x)\le g(x)$ for all $x\in \lambda$.

For a set $X$ by $\IR^X$ we denote the poset of all real-valued functions on $X$, endowed with the partial order $\le$ defined by $f\le g$ iff $f(x)\le g(x)$ for all $x\in X$. If $X$ is a topological space, then the poset $\IR^X$ contains the subposet $C(X)$ consisting of all continuous real-valued functions on $X$. By $C_p(X)$ (resp. $C_k(X)$) we denote the space $C(X)$ endowed with the topology of pointwise convergence (resp. the compact-open topology). %If $X$ is a uniform space, then the poset $C(X)$ contains the subposet $C_u(X)$ consisting of all uniformly continuous real-valued functions on $X$. It is clear that $C_u(X)\subset C(X)\subset\IR^X$. The space $C_u(X)$ will be endowed with the topology inherited from the Tychonoff product topology of $\IR^X$.

For a point $x$ of a topological space $X$ by $\Tau_x(X)$ we denote the poset of all neighborhoods of $x$ in $X$, endowed with the partial order of converse inclusion ($U\le V$ iff $V\subset U$).
Observe that a topological space $X$ has a neighborhood $\w^\w$-base at a point $x\in X$ if and only if  $\w^\w\succcurlyeq\Tau_x(X)$. %where $\w^\w$ is the poset of all functions from $\w$ to $\w$, endowed with the partial order $\le$ defined by $f\le g$ iff $f(n)\le g(n)$ for all $n\in\w$.

In the sequel we equip every topological group $G$ with the two-sided uniformity generated by the base consisting of entourages $\{(x,y)\in G:y\in xU\cap Ux\}$ where $U\in\Tau_e(G)$. Here by $e$ we denote the identity element of the group $G$.
In the case of an Abelian group $G$ we use symbol $0$ for the neutral element of $G$.
If $X$ is a subset of a topological group $G$ and $e \in X$ then $\Tau_e(X)$ denotes the poset of all neighborhoods of $e$ in $X$, as above.

In particular, we equip every topological vector space $L$ with the natural uniformity generated by the base consisting of entourages $\{(x,y)\in L:x-y\in U\}$ where $U\in\Tau_0(L)$. Here  $\Tau_0(L)$ denotes the poset of all neighborhoods of zero in $L$.

%For a set $X$ by $\IR^X$ we denote the set of all real-valued functions on $X$ endowed with the pointwise partial order $\le$ defined by $f\le g$ iff $f(x)\le g(x)$ for all $x\in X$. If $X$ is a topological space, then $C(X)$ stands for the poset of all continuous real-valued functions on $X$ endowed with the partial order inherited from $\IR^X$.

Given a uniform space $X$ by $\U_X$ we denote the uniformity of $X$ endowed with the partial order of converse inclusion ($U\le V$ iff $V\subset U$). For an entourage $U\in\U_X$ and a point $x\in X$ by $U[x]$ we denote the $U$-ball $\{y\in X:(x,y)\in U\}\subset X$ centered at $x$.
%A uniform space $X$ is called {\em $\w$-narrow} if $\cov^\sharp(X;\U_X)\le\w_1$, i.e., for any entourage $U\in\U_X$ there exists a countable subset $C\subset X$ such that $X=U[C]$.
Each uniform space $X$ will be endowed with the (Tychonoff) topology consisting of all subsets $W\subset X$ such that for every $x\in W$ there exists an entourage $U\in\U_X$ such that $U[x]\subset W$.

Conversely, each Tychonoff space $X$ will be equipped with the {\em  universal uniformity\/} $\U_X$ generated by the base consisting of entourages $[d]_{<1}=\{(x,y):\allowbreak d(x,y)<1\}$ where $d$ runs over all continuous pseudometrics on $X$.

Given a uniform space $X$ by $C_u(X)$ we denote the space of all real-valued uniformly continuous functions on $X$. Here we consider $C_u(X)$ as a posed endowed with the partial order inherited from the poset $\IR^X$ of all real-valued functions on $X$. We shall endow the space $C_u(X)$ with the topology of pointwise convergence inherited from the Tychonoff power $\IR^X$.
Observe that for a Tychonoff space $X$ equipped with the universal uniformity  the poset $C_u(X)$ coincides with $C(X)$ (or $C_p(X)$ if $C_u(X)$ is considered as a topological space).

A subset $B$ of a uniform space $X$ is called {\em functionally bounded} if for each uniformly continuous function $f:X\to\IR$ the set $f(B)$ is bounded in $\IR$; $B$ is {\em $\sigma$-bounded} if $B$ is a countable union of functionally bounded sets in $X$. Observe that any convex (more generally, star-like) subset of a Banach space is $\sigma$-bounded.

A uniform space $X$ is called {\em universal} if each continuous function $f:X\to M$ to a metric space $M$ is uniformly continuous. In this case the uniformity $\U_X$ of $X$ coincides with the universal uniformity of the Tychonoff space $X$. The notion of universality can be parametrized in the following way.

Let $\kappa$ be a cardinal. A function $f:X\to Y$ between uniform spaces is called {\em $\kappa$-continuous} if for any entourage $U\in\U_Y$ there exists a subfamily $\V\subset\U_X$ of cardinality $|\V|\le\kappa$ such that for every $x\in X$ there exists an entourage $V\in\V$ such that  $f(V[x])\subset U[f(x)]$.  It is clear that each $\kappa$-continuous function $f:X\to Y$ between uniform spaces is continuous with respect to the topologies generated by the uniformities.

A uniform space $X$ is called {\em $\kappa$-universal} if each $\kappa$-continuous function $f:X\to M$ to a metric space $M$ of density $\kappa$ is uniformly continuous. It is clear that a uniform space $X$ is universal if and only if it is $\kappa$-universal for every cardinal $\kappa$.

By $C_\w(X)$ we denote the poset  of all $\w$-continuous real-valued functions on a uniform space. It follows that  $C_u(X)\subset C_\w(X)\subset C(X)\subset\IR^X$. If the space $X$ is Lindel\"of, then every continuous real-valued function on $X$ is $\w$-continuous, so $C_\w(X)=C(X)$.

We define a uniform space $X$ to be {\em $\IR$-universal} if $C_\w(X)=C_u(X)$ (i.e., each $\w$-continuous real-valued function on $X$ is uniformly continuous). It is clear that each $\w$-universal uniform space is $\IR$-universal. In particular, Tychonoff space endowed with its universal uniformity is $\IR$-universal. For each uniform space $X$ we get the implications:

\centerline{universal $\Ra$ $\w_1$-universal $\Ra$ $\w$-universal $\Ra$ $\IR$-universal.}
\smallskip

A uniform space $X$ is called {\em $\IR$-complete} if the canonical map $\delta:X\to\IR^{C_u(X)}$ assigning to each point $x\in X$ the Dirac measure $\delta_x:C_u(X)\to\IR$, $\delta_x:f\mapsto f(x)$, is a closed topological embedding. Observe that a Tychonoff space $X$ endowed with its universal uniformity is $\IR$-complete if and only if it is Hewitt complete. By \cite[3.11.6]{Ban}, the completion $\bar X$ of any $\IR$-universal $\w$-narrow uniform space $X$ is $\IR$-complete.
 By \cite[3.9.5]{Ban}, a uniform space $X$ is compact if and only if it is functionally bounded and $\IR$-complete. An example of a complete uniform space which is not $\IR$-complete is any non-compact closed bounded convex subset  of a Banach space, see \cite[3.11.7]{Ban}.

A function $f:X\to \IR$ on a uniform space will be called {\em locally $\w$-bounded} if there exists a countable family of entourages $\V\subset\U_X$ such that for every $x\in X$ there exists an entourage $V\in\V$ such that $\sup f(V[x])<\infty$.  By $B_\w(X)$ we denote the subset of $\IR^X$ consisting of all locally $\w$-bounded real-valued functions on $X$. It is clear that each $\w$-continuous function is locally $\w$-bounded and hence $C_\w(X)\subset B_\w(X)\subset \IR^X$. The following lemma shows that  $C_\w(X)$ is cofinal in the poset $B_\w(X)$.

\begin{lemma}\label{l:CwB}  For any locally $\w$-bounded function $\varphi:X\to\IR$ on a uniform space $X$ there exists an $\w$-continuous function $\psi:X\to\IR$ such that $\psi\ge \varphi$.
\end{lemma}

\begin{proof}  Given any locally $\w$-bounded function $\varphi:X\to\IR$, find a countable family of entourages $\{U_n\}_{n\in\w}\subset\U_X$ such that for every point $x\in X$ there is $n\in\w$ such that the set $\varphi(U_n[x])$ is bounded in $\IR$. Choose a sequence of entourages $(V_n)_{n\in\w}\subset\U_X$ such that $V_{n+1}^3\subset V_n=V_n^{-1}\subset\bigcap_{k\le n}U_n$ for every $n\in\w$. By Theorem~8.1.10 \cite{Eng}, there exists a pseudometric $d$ on $X$ such that
$$\{(x,y):d(x,y)<2^{-n}\}\subset V_n\subset \{(x,y):d(x,y)\le 2^{-n}\}$$for every $n\in\IN$. By the choice of the sequence $(U_n)_{n\in\w}$, for every $x\in X$ there exist numbers $l_x,n_x\in\w$ such that $\varphi(U_{l_x}[x])\subset[0,n_x]$.  Endow the space $X$ with the topology $\tau$ generated by the pseudometric $d$. By the Stone Theorem 4.4.1 \cite{Eng}
(on the paracompactness of metrizable spaces), there exists a locally finite open cover $\V\subset\tau$ of the pseudometric space $(X,d)$ such that each set $V\in\V$ is contained in some ball $U_{l_x}[x]$. For every $V\in\V$ let $n_V=\min\{n_x:x\in X,\;V\subset U_{l_x}[x]\}$. It follows that $\varphi(V)\subset [0,n_V]$. Since the cover $\V$  is locally finite, there exists a locally finite open cover  $\mathcal W$ of the pseudometric space $(X,d)$ such that for every set $W\in\mathcal W$ the family $\St(W,\V)=\{V\in\V:W\cap V\ne\emptyset\}$ is finite and hence the number $m_W=\max\{n_V:V\in\St(W,\V)\}$ is well-defined.

By the paracompactness of the pseudometric space $(X,d)$, there exists a partition of the unity $(\lambda_W:X\to[0,1])_{W\in\mathcal W}$ such that for every $W\in\mathcal W$ the support $\supp(\lambda_W)=\lambda_W^{-1}((0,1])$ is contained in $W$. Then the function $\psi:X\to\IR$ defined by $\psi(x)=\sum_{W\in\w}\lambda_W(x)m_W$ is well-defined and $\w$-continuous (being continuous with respect to the topology generated by the pseudometric $d$). It remains to check that $\varphi\le\psi$. Take any point $x\in X$ and find a set $V\in\V$ containing $x$. It follows that $\varphi(x)\in\varphi(V)\subset [0,n_V]$. Observe that each set $W\in\mathcal W$ that contains $x$ has $m_W\ge n_V\ge \varphi(x)$.
Then $\psi(x)=\sum_{W\in\mathcal W}\lambda_W(x)m_W\ge \varphi(x)$ as well.
\end{proof}

A uniform space $X$ is called
\begin{itemize}
\item {\em discrete} if its uniformity $\U_X$ contains the diagonal $\Delta_X$ of the square $X\times X$;
\item a {\em uniform $P_u$-space} if for any countable family $\V\subset\U_X$ the intersection $\bigcap\V$ belongs to $\U_X$.
\end{itemize}
Observe that a uniform space $X$ is a $P_u$-space if and only if $\add(X)>\w$.

Finally, we discuss some cardinal characteristics of posets related to cofinal and bounded sets. A subset $A$ of a poset $P$ is called
\begin{itemize}
\item {\em cofinal} if $\forall x\in X\;\exists a\in A\;\;(x\le a)$;
\item {\em bounded} if $\exists x\in X\;\forall a\in A\;\;(a\le x)$.
\end{itemize}
The {\em cofinality} $\cof(P)$ of a poset $P$ is the smallest cardinality of a cofinal subset of $X$. For an unbounded poset $P$ its {\em additivity} $\add(P)$ is the smallest cardinality of an unbounded subset of $P$. For a non-discrete uniform space $X$ its additivity $\add(X)$ coincides with the additivity $\add(\U_X)$ of its uniformity $\U_X$ (considered as a poset endowed with the partial order of reverse inclusion).

It is well-known \cite[513C]{F-MT} that for any unbounded poset $P$ the cardinal $\add(P)$ is regular and $\add(P)\le\cf(\cof(P))$. By \cite[2.3.2]{Ban}, $P\succcurlyeq\add(P)$ and $P\succcurlyeq\cof(P)$.

For the set $\w^\w$ endowed with the partial preorder $\le^*$ (defined by $x\le^* y$ iff  the set $\{n\in\w:x(n)\not\le^* y(n)\}$ is finite) the cardinals $\mathfrak b=\add(\w^\w,\le^*)$ and $\mathfrak d=\cof(\w^\w,\le^*)$ play an important role in the Modern Set Theory, see \cite{Douwen}, \cite{Vaug}.

Following \cite[\S2.4]{Ban}, for each cardinal $\kappa>1$ denote by $\mathfrak e(\kappa)$ the smallest cardinal $\lambda$ such that $\w^\w\not\succcurlyeq \kappa^\lambda$. Therefore, $\w^\w\not\succcurlyeq\kappa^{\mathfrak e(\kappa)}$ but $\w^\w\succcurlyeq\kappa^\lambda$ for any cardinal $\lambda<\mathfrak e(\kappa)$.
The following properties of the function $\mathfrak e(\kappa)$ has been established in \cite[2.4.1]{Ban}.

\begin{proposition}\label{p:e(X)} Let $\kappa$ be an infinite cardinal. Then
\begin{enumerate}
\item $\mathfrak e(\kappa)=\mathfrak e(\cof(\kappa))$.
\item If $\cof(\kappa)=\w$, then $\mathfrak e(\kappa)=\w_1$.
\item If $\mathfrak e(\kappa)>1$, then $\mathfrak e(\kappa)\ge\cf(\kappa)\in  \{\w\}\cup[\mathfrak b,\mathfrak d]$.
\item $\mathfrak e(\kappa)\in\{1,\w_1\}\cup[\mathfrak b,\mathfrak d]$.
\end{enumerate}
\end{proposition}

Proposition~\ref{p:e(X)} implies that the cardinal
$$\mathfrak e^\sharp=\sup\{\kappa^+:\kappa=\cf(\kappa)\mbox{ and }\w^\w\succcurlyeq\kappa^\kappa\}$$is well-defined and belongs to the set $\{\w_1\}\cup(\mathfrak b,\mathfrak d]$. In particular, $\mathfrak e^\sharp=\w_1$ if $\mathfrak b=\mathfrak d$.  The assumption $\mathfrak b=\mathfrak d$ is strictly weaker than $\mathfrak e^\sharp=\w_1$; and the strict inequality $\mathfrak e^{\sharp}>\w_1$  is consistent with ZFC, see \cite[2.4.12]{Ban}. The definition of the cardinal $\mathfrak e^\sharp$ implies:

\begin{proposition}\label{p:e=w1} If $\mathfrak e^\sharp=\w_1$, then $\mathfrak e(\kappa)\in\{1,\kappa\}$ for any regular uncountable cardinal $\kappa$.
\end{proposition}

\section{Uniform spaces with an $\w^\w$-base}

In this section we survey the results of the paper \cite{Ban} describing the properties of uniform spaces whose uniformity has an $\w^\w$-base. Such uniform spaces are called {\em $\w^\w$-based}.
As shown in \cite{Ban}, $\w^\w$-based uniform spaces have many properties, typical for generalized metric spaces. So, first we recall the definitions of such properties. We start with local properties of topological spaces.

A topological space $X$ is defined to be
\begin{itemize}
\item {\em first-countable} at a point $x\in X$ if $X$ has a countable neighborhood base at $x$;
\item {\em Fr\'echet-Urysohn at} a point $x\in X$ if for each set $A\subset X$ with $x\in\bar A$ there is a sequence $\{a_n\}_{n\in\w}\subset A$ converging to $x$;
%\item {\em sequential} if for each non-closed set $A\subset X$ there exists a sequence $\{a_n\}_{n\in\w}\subset A$, convergent to some point $x\in X$;
%\item {\em strong Fr\'echet at} a point $x\in X$ if for any decreasing sequence $(A_n)_{n\in\w}$ of subsets of $X$ with $x\in\bigcap_{n\in\w}\bar A_n$ there exists a sequence $(x_n)_{n\in\w}\in\bigcap_{n\in\w}A_n$ that converges to $x$;
\item a {\em $q$-space at} a point $x\in X$ if there exists a sequence $(U_n)_{n\in\w}$ of neighborhoods of $x$ such every sequence $(x_n)_{n\in\w}\in\prod_{n\in\w}U_n$ accumulates at some point $x_\infty\in X$;
%\item a {\em $q$-space} if $X$ is a $q$-space at each point $x\in X$;
\item {\em countably tight at} a point $x\in X$ if each subset $A\subset X$ with $x\in\bar A$ contains a countable subset $B\subset A$ such that $x\in\bar B$;
\item {\em countably fan tight} at a point $x\in X$ if for any decreasing sequence $(A_n)_{n\in\w}$ of subsets of $X$ with $x\in\bigcap_{n\in\w}\bar A_n$ there exists a sequence $(F_n)_{n\in\w}$ of finite subsets $F_n\subset A_n$, $n\in\w$, accumulating at $x$;
\item {\em countably ofan tight} at a point $x\in X$ if for any decreasing sequence $(A_n)_{n\in\w}$ of open subsets of $X$ with $x\in\bigcap_{n\in\w}\bar A_n$ there exists a sequence $(F_n)_{n\in\w}$ of finite subsets $F_n\subset A_n$, $n\in\w$, accumulating at $x\in X$;
\item  {\em first-countable} (resp. {\em Fr\'echet-Urysohn, countably tight, countably fan tight, countably ofan tight, a $q$-space}) if $X$ is first-countable (resp. Fr\'echet-Urysohn,  countably tight, countably fan tight, countably ofan tight, a $q$-space) at each point $x\in X$.
\end{itemize}
Fr\'echet-Urysohn and $q$-spaces are well-known in General Topology, see \cite{Eng}, \cite{EG}.  The countable fan tightness was introduced by Arhangelskii \cite{Ar86} and is well-known and useful property in $C_p$-theory \cite{Arch}. Its ``open'' modification was introduced by Sakai \cite{MSak} as the property $(\#)$. In \cite{BanP0} the countable ofan tightness was called the countable fan open-tightness. The countable (o)fan tightness is equivalent to the (o)fan $\mathsf s$-tightness introduced in \cite[1.2.1]{Ban}.
\smallskip

  The following diagram describes the implications between some local properties of topological spaces.
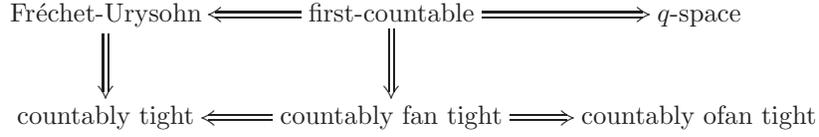
{%\small
$$\xymatrix{
\mbox{Fr\'echet-Urysohn}\ar@{=>}[d]&\mbox{first-countable}\ar@{=>}[l]\ar@{=>}[r]\ar@{=>}[d]&\mbox{$q$-space}\\
\mbox{countably tight}&\mbox{countably fan tight}\ar@{=>}[r]\ar@{=>}[l]&\mbox{countably ofan tight}
}
$$}

The following characterization of first-countable $\w^\w$-based uniform spaces was proved in \cite[6.6.1]{Ban}.

\begin{theorem} For a Tychonoff space $X$ whose topology is generated by an $\w^\w$-based uniformity and a point $x\in X$ the following conditions are equivalent:
\begin{enumerate}
\item $X$ is first countable at $x$;
\item $X$ is a $q$-space at $x$;
\item $X$ is countably fan tight at $x$;
\item $X$ is countably ofan tight at $x$.
\end{enumerate}
\end{theorem}

A family $\mathcal N$ of subsets of a topological space $X$ is called
\begin{itemize}\itemsep=2pt
\item a {\em network} if for any open set $U\subset X$ and point $x\in U$ there exists a set $N\in\mathcal N$ with $x\in N\subset U$;
\item a {\em $k$-network} if for any open set $U\subset X$ and compact subset $K\subset U$ there exists a finite subfamily $\F\subset \mathcal N$ with $K\subset\bigcup\F\subset U$;
\item a {\em $\C$-network} for a family $\C$ of subsets of $X$ if for any set $C\in\C$ and any open neighborhood $U\subset X$ of $C$ there is a set $N\in\mathcal N$ such that $C\subset N\subset U$;
\item a {\em $\mathsf{cs}^*$-network} if for any open set $U\subset X$ and any sequence $(x_n)_{n\in\w}\in X^\w$ converging to a point $x\in U$ there exists a set $N\in\mathcal N$ such that $x\in N\subset O_x$ and $N$ contains infinitely many points $x_n$, $n\in\w$;
\item a {\em Pytkeev$^*$-network}  if any open set $U\subset X$ and any sequence $(x_n)_{n\in\w}\in X^\w$ accumulating at a point $x\in U$ there exists a set $N\in\mathcal N$ such that $x\in N\subset O_x$ and $N$ contains infinitely many points $x_n$, $n\in\w$;
 \item a {\em Pytkeev network} if for any neighborhood $O_x\subset X$ of $x$ and any set $A\subset X$ containing $x$ in its closure there exists a set $N\in\mathcal N$ such that $x\in N\subset O_x$, $N\cap A\ne\emptyset$, and moreover $N\cap A$  is infinite if $A$ accumulates at $x$.
 \end{itemize}

 A regular topological space $X$ is called
 \begin{itemize}\itemsep=2pt
  \item {\em cosmic} if $X$ has a countable network;
  \item a {\em $\sigma$-space} if $X$ has a $\sigma$-discrete network;
%\item a {\em strong $\sigma$-space} if $X$ has a strongly $\sigma$-discrete network;
  \item a {\em $\Sigma$-space} if $X$ has a $\sigma$-discrete $\C$-network for some cover $\C$ of $X$ by closed countably compact subspaces;
 \item an {\em $\aleph_0$-space} if $X$ has a countable $k$-network (equivalently, $X$ has a countable $\mathsf{cs}^*$-network);
 \item an {\em $\aleph$-space} if $X$ has a $\sigma$-discrete $k$-network (equivalently, $X$ has a $\sigma$-discrete $\mathsf{cs}^*$-network);
 \item a {\em $\mathfrak P_0$-space} if $X$ has a countable Pytkeev network (equivalently, $X$ has a countable Pytkeev$^*$ network);
 \item a {\em $\mathfrak P$-space} if $X$ has a $\sigma$-discrete Pytkeev network;
 \item a {\em $\mathfrak P^*$-space} if $X$ has a $\sigma$-discrete Pytkeev$^*$ network;
  \item a {\em La\v snev space} if $X$ is the image of a metrizable space under a closed continuous map;
 \item a {\em $w\Delta$-space} if $X$ admits a sequence $(\U_n)_{n\in\w}$ of open covers such that for every point $x\in X$, any sequence $(x_n)_{n\in\w}\in \prod_{n\in\w}\St(x,\U_n)$ has an accumulation  point in $X$;
 \item an {\em $M$-space} if $X$ admits a sequence $(\U_n)_{n\in\w}$ of open covers such that each $\U_{n+1}$ star-refines $\U_n$ and for every point $x\in X$ any sequence $(x_n)_{n\in\w}\in \prod_{n\in\w}\St(x,\U_n)$ has an accumulation  point in $X$;
 \item a space with a {\em $G_\delta$-diagonal} if the diagonal $\Delta_X=\{(x,y)\in X\times X:x=y\}$ is a $G_\delta$-set in $X\times X$;
\item {\em submetrizable} if $X$ admits a continuous injective map to a metrizable space;
\item {\em $\sigma$-metrizable} if $X$ has a countable cover by closed metrizable subspaces;
\item a {\em closed-$\bar G_\delta$ space} if each closed subset $F$ of $X$ is a {\em $\bar G_\delta$-set} in $X$, i.e., $F$ can be written as $F=\bigcap_{n\in\w}U_n=\bigcap_{n\in\w}\overline{U}_n$ for some sequence $(U_n)_{n\in\w}$ of open sets in $X$.
    \end{itemize}
For every topological space these properties relate as follows.

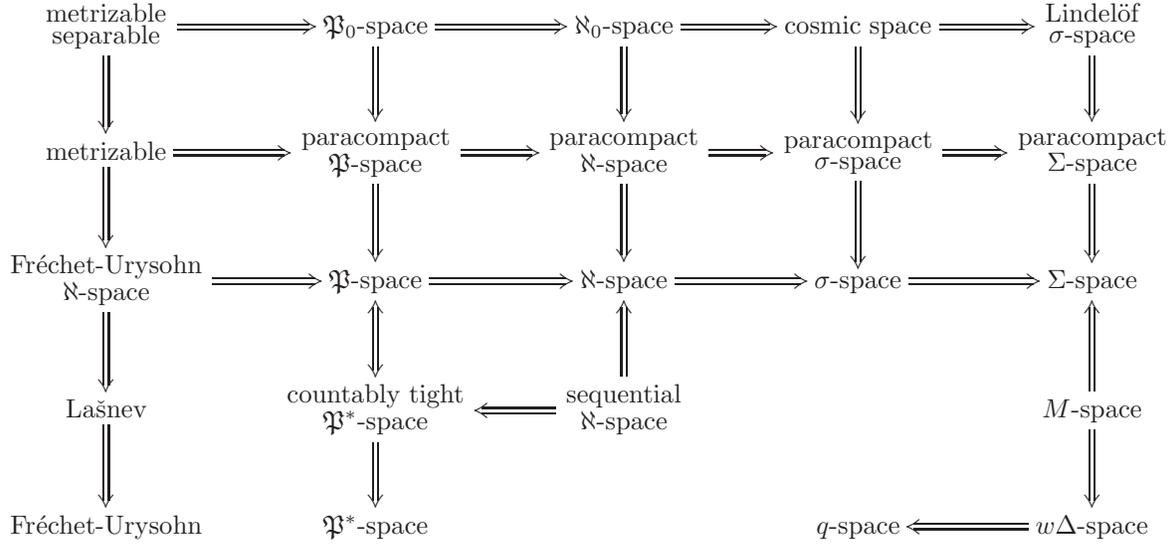
{%\small
 $$
 \xymatrix{
\mbox{metrizable}\atop\mbox{separable}\ar@{=>}[r]\ar@{=>}[d]&\mbox{$\mathfrak P_0$-space}\ar@{=>}[r]\ar@{=>}[d]&\mbox{$\aleph_0$-space}\ar@{=>}[r]\ar@{=>}[d]&\mbox{cosmic space}\ar@{=>}[d]\ar@{=>}[r]&\mbox{Lindel\"of}\atop\mbox{$\sigma$-space}\ar@{=>}[d]\\
\mbox{metrizable}\ar@{=>}[r]\ar@{=>}[d]& \mbox{paracompact}\atop\mbox{$\mathfrak P$-space}\ar@{=>}[r]\ar@{=>}[d]&\mbox{paracompact}\atop\mbox{$\aleph$-space}\ar@{=>}[r]\ar@{=>}[d]&\mbox{paracompact}\atop\mbox{$\sigma$-space}\ar@{=>}[d]\ar@{=>}[r]&\mbox{paracompact}\atop\mbox{$\Sigma$-space}\ar@{=>}[d]\\
\mbox{Fr\'echet-Urysohn}\atop\mbox{$\aleph$-space}\ar@{=>}[r]\ar@{=>}[d] &\mbox{$\mathfrak P$-space}\ar@{=>}[r]&\mbox{$\aleph$-space}\ar@{=>}[r]&\mbox{$\sigma$-space}\ar@{=>}[r]&\mbox{$\Sigma$-space}\\
\mbox{La\v snev}\ar@{=>}[d]&\mbox{countably tight}\atop\mbox{$\mathfrak P^*$-space}\ar@{<=>}[u]\ar@{=>}[d]&\mbox{sequential}\atop\mbox{$\aleph$-space}\ar@{=>}[u]\ar@{=>}[l]&&\mbox{$M$-space}\ar@{=>}[d]\ar@{=>}[u]\\
\mbox{Fr\'echet-Urysohn}&\mbox{$\mathfrak P^*$-space}&&\mbox{$q$-space}&\mbox{$w\Delta$-space}\ar@{=>}[l]
    }
 $$
 }
More information on these and other classes of generalized metric spaces can be found in \cite{Grue}, \cite{BanP0}, \cite{GK-P}, \cite{Ban}, \cite{Ban2}.
\vskip5pt

Uniform spaces with an $\w^\w$-base have the following topological properties, proved in \cite[6.6.7,6.8.5]{Ban}.

\begin{theorem}\label{t:Sigma} Let $X$ be a Tychonoff space whose topology is generated by an $\w^\w$-based uniformity $\U$.
\begin{enumerate}
\item If $X$ is a $w\Delta$-space, then $X$ has a $G_\delta$-diagonal.
\item $X$ is a $\Sigma$-space if and only if $X$ is a $\sigma$-space.
\item If $X$ is a (countably tight) collectionwise $\Sigma$-space, then $X$ is a paracompact $\mathfrak P^*$-space (and a $\mathfrak P$-space).
\item Each totally bounded subsets of the uniform space $(X,\U)$ is metrizable and separable.
\end{enumerate}
\end{theorem}

The following metrization theorem is proved in \cite[6.6.2]{Ban}.

\begin{theorem}\label{t:metr-ww} For a Tychonoff space $X$ whose topology is generated by an $\w^\w$-based uniformity the following conditions are equivalent:
\begin{enumerate}
\item $X$ is metrizable;
\item $X$ is first-countable closed-$\bar G_\delta$ $T_0$-space;
\item $X$ is an $M$-space;
\item $X$ is first-countable collectionwise normal $\Sigma$-space.
%\item $X$ is first-countable strong $\sigma$-space.
\end{enumerate}
\end{theorem}

Next, we recall some information related to the analycity of topological spaces. A topological space $X$ is {\em analytic} if $X$ is a continuous image of a Polish space. A non-empty topological space is analytic if and only if it is a continuous image of the Polish space $\w^\w$. A topological space $X$ is {\em $K$-analytic} if $X=\bigcup_{\alpha\in \w^\w}K_\alpha$ for some family $(K_\alpha)_{\alpha\in \w^\w}$ of compact subsets of $X$, which is {\em upper semi-continuous} in the sense that for any open set $U\subset X$ the set $\{\alpha\in\w^\w:K_\alpha\subset U\}$ is open in $\w^\w$. By \cite[1.5.1]{Ban}, a Tychonoff space $X$ is analytic if and only if $X$ is a submetrizable $K$-analytic space.

Following \cite{kak}, we say that a topological space $X$ has a {\em compact resolution} if $X=\bigcup_{\alpha\in \w^\w}K_\alpha$ for some family $(K_\alpha)_{\alpha\in\w^\w}$ of compact subsets of $X$ such that $K_\alpha\subset K_\beta$ for all $\alpha\le\beta$ in $\w^\w$. By Proposition 3.10(i) (and 3.13) in \cite{kak}, a (Dieudonn\'e complete) Tychonoff space $X$ has a compact resolution if (and only if) $X$ is $K$-analytic.
It is well-known that each $K$-analytic space $X$ is Lindel\"of and hence has {\em countable extent}. The latter means that all closed discrete subspaces in $X$ are countable. A topological space $X$ has {\em countable discrete cellularity} if each discrete family $\F$ of non-empty open sets in $X$ is at most countable.

For a topological space $X$ by $X'$ we denote the set of non-isolated points of $X$, by $X^{\prime \mathsf s}$ the set of accumulating points of countable subsets in $X$, and by $X^{\prime P}$ the set of points $x\in X$ which are not $P$-points in $X$ (which means that some $G_\delta$-set $G\subset X$ containing $x$ is not a neighborhood of $x$). A non-isolated point $x$ is a $P$-point iff $\add(\Tau_x(X))>\w$.
It is clear that $X^{\prime \mathsf s}\subset X^{\prime\mathsf P}\subset X'$. For countably tight spaces $X$ we get the equality $X'=X^{\prime\mathsf P}=X^{\prime \mathsf s}$.

The following theorem is proved in \cite[6.10.1, 7.5.1]{Ban}.

\begin{theorem}\label{t:looong} Let $X$ be an $\w^\w$-based uniform space.\newline
The following conditions \textup{(1)--(7)} are equivalent:
\begin{enumerate}
\item[\textup{(1)}] $X$ is separable;
\item[\textup{(2)}] $X$ is cosmic;
\item[\textup{(3)}] $X$ is an $\aleph_0$-space;
\item[\textup{(4)}] $X$ is a $\mathfrak P_0$-space.
\item[\textup{(5)}] $X$  contains a dense $\Sigma$-subspace with countable extent;
\item[\textup{(6)}] $C_u(X)$ is analytic;
\item[\textup{(7)}] $C_u(X)$ is cosmic.
\end{enumerate}
The equivalent conditions $(1)$--$(7)$ imply the condition
\begin{enumerate}
\item[\textup{(8)}] the uniform space $X$ contains a dense $\sigma$-bounded subset.
\end{enumerate}
The condition $(8)$ implies the equivalent conditions \textup{(9)--(13)}:
\begin{enumerate}
\item[\textup{(9)}] the poset $C_u(X)$ is $\w^\w$-dominated;
\item[\textup{(10)}] $C_u(X)$ is $\w^\w$-dominated in $\IR^X$;
\item[\textup{(11)}] for some dense subspace $Z\subset X$ the set $\{f|Z:f\in C_u(X)\}$ is $\w^\w$-dominated in $\IR^Z$;
\item[\textup{(12)}] $C_u(X)$ is $K$-analytic;
\item[\textup{(13)}] $C_u(X)$ has a compact resolution.
\end{enumerate}
If the completion of $X$ is  $\IR$-complete, then the conditions  $(1)$--$(13)$ are equivalent.
\newline
If the uniform space $X$ is $\IR$-universal, then the conditions $(1)$--$(13)$ are equivalent to any of the conditions:
\begin{enumerate}
\item[\textup{(14)}] $|X\setminus X^{\prime\mathsf s}|\le\w$;
\item[\textup{(15)}] $X$ is $\sigma$-compact;
\item[\textup{(16)}] $C_p(X)$ is cosmic;
\item[\textup{(17)}] $C_p(X)$ is analytic.
\item[\textup{(18)}] $C_p(X)$ is $K$-analytic;
\item[\textup{(19)}] $C_p(X)$ has a compact resolution.
\item[\textup{(20)}] $C_k(X)$ is cosmic;
\item[\textup{(21)}] $C_k(X)$ is analytic.
\item[\textup{(22)}] $C_k(X)$ is $K$-analytic;
\item[\textup{(23)}] $C_k(X)$ has a compact resolution.
\end{enumerate}
Under $\w_1<\mathfrak b$ the conditions $(1)$--$(7)$ are equivalent to:
\begin{enumerate}
\item[\textup{(24)}] $X$ is Lindel\"of;
\item[\textup{(25)}] $X$ has countable discrete cellularity;
\item[\textup{(26)}] the uniform space $X$ is $\w$-narrow (i.e., $\cov^\sharp(X)\le\w_1$).
\end{enumerate}
If $\w_1<\mathfrak b$ and the uniform space $X$ is $\w_1$-universal, then the conditions \textup{(1)--(26)} are equivalent to:
\begin{enumerate}
\item[\textup{(27)}] $|X\setminus X^{\prime\mathsf P}|\le\w$.
\end{enumerate}
%Under $\w_1=\mathfrak b$ the conditions \textup{(1)--(8)} are not equivalent to $(16)$.
%Under $\w_1=\mathfrak b$ the conditions $(1)$--$(7)$ are not equivalent to $(24)$.
\end{theorem}

By \cite[6.3.9]{Ban}, the function space $C_k(\w_1)$  is Lindel\"of, non-separable and under $\w_1=\mathfrak b$ is an  $\w^\w$-based uniform space (see \cite[Theorem 3]{FKLS}). Here $C_k(\w_1)$ is the space of continuous real-valued functions on the cardinal $\w_1=[0,\w_1)$ endowed with the order topology. By \cite[6.3.8]{Ban}, the closed subgroup $C_k(\w_1,\IZ)\subset C_k(\w_1)$ of integer-valued functions on $\w_1$ is Lindel\"of, not separable, and under $\w_1=\mathfrak b$ is universally $\w^\w$-based. We recall that a Tychonoff space $X$ is {\em universally $\w^\w$-based} if the universal uniformity $\U_X$ of $X$ has an $\w^\w$-base. For a universally $\w^\w$-based Tychonoff space $X$ the conditions of (1)--(23) of Theorem~\ref{t:looong} are equivalent.
The universally $\w^\w$-based Lindel\"of non-separable space $C_k(\w_1,\IZ)$ shows that under $\w_1=\mathfrak b$ the conditions (1)--(23) of Theorem~\ref{t:looong} are not equivalent to the condition (24).  %A universally $\w^\w$-based space satisfying these equivalent conditions will be called a {\em small universally $\w^\w$-space}.

Universally $\w^\w$-based spaces were studied in \cite[Chapter 7]{Ban} where it is shown that such spaces are close to being $\sigma'$-compact. A topological space $X$ is defined to be {\em $\sigma'$-compact} if the set $X'$ of non-isolated points of $X$ is $\sigma$-compact.

By \cite[3.14]{LPT} (and \cite[7.8.7]{Ban}), a metrizable space $X$ is universally $\w^\w$-based if (and only if) it is $\sigma'$-compact. In Proposition~\ref{p:Rquot} we shall improve this result proving that a Tychonoff space $X$ is universally $\w^\w$-based if $X$ is a quotient image of a metrizable $\sigma'$-compact space. For La\v snev spaces (i.e., images of metrizable spaces under closed maps) we have the following characterization proved in \cite[7.8.10,8.3.1]{Ban}.

\begin{theorem}\label{t:Lasnev} For a La\v snev space $X$ the following conditions are equivalent:
\begin{enumerate}
\item $X$ is universally $\w^\w$-based;
\item $X$ is the image of a metrizable $\sigma'$-compact space under a closed continuous map;
\item $X$ is the image of a metrizable $\sigma'$-compact space under a quotient map.
\item $X$ is the image of a metrizable $\sigma'$-compact space under an $\IR$-quotient map.
\end{enumerate}
\end{theorem}

A map $f:X\to Y$ between topological spaces is called {\em $\IR$-quotient} if for any function $\varphi:Y\to\IR$ the continuity of $\varphi$ is equivalent to the continuity of the composition $\varphi\circ f:X\to\IR$.

Following \cite{Ban}, we define a subset $B\subset X$ of a topological space $X$ to be {\em $\w$-Urysohn} if each infinite closed discrete set $D\subset B$ in $X$ contains an infinite strongly discrete subset $S\subset D$ in $X$. A subset $S\subset D$ is called {\em strongly discrete} if each point $x\in S$ has a neighborhood $O_x\subset X$ such that the family $(O_x)_{x\in S}$ is discrete in $X$, which means that each point $z\in X$ has a neighborhood $U_z$ that meets at most one set $O_x$, $x\in S$. It is easy to see that each subset of a normal space is $\w$-Urysohn.

 The following theorem is proved in \cite[7.2.1,7.4.1,7.4.2,7.8.19,7.4.3]{Ban}.

\begin{theorem}\label{t:uwwb} If $X$ is a universally $\w^\w$-based space, then
\begin{enumerate}
\item the set $X^{\prime\mathsf s}$ is closed and $\sigma$-bounded in $X$;
\item the set $X^{\prime\mathsf P}$ is closed and $\w$-narrow in $X$;
\item if $\w_1<\mathfrak b$, then $|X^{\prime\mathsf P}\setminus X^{\prime\mathsf s}|\le\w$ and the set $X^{\prime\mathsf P}$ is $\sigma$-bounded in $X$;
\item $\cov^\sharp(X^{\prime\mathsf P};\U_X)\le\w_1$ and $\cov^\sharp(X';\U_X)\le\mathfrak d$;
\item $X^{\prime\mathsf s}$ is a $\sigma$-metrizable $\mathfrak P_0$-space;
\item a closed subset $F\subset X$ is $\sigma$-compact if $|F\setminus X^{\prime \mathsf s}|\le\w$ and $F$ is $\w$-Urysohn or a $\bar G_\delta$-set in $X$;
\item $X$ is $\sigma'$-compact if and only if $|X'\setminus X^{\prime\mathsf s}|\le\w$;
\item Under $\w_1<\mathfrak b$ the space $X$ is a $\sigma$-space if and only if $X'$ is a $\sigma$-compact $G_\delta$-set in $X$.
\end{enumerate}
\end{theorem}

Finally we discuss some properties of uniform spaces $X$ with $\w^\w$-dominated function spaces $C_u(X)$. We recall that a uniform space $X$ is {\em $\sigma$-bounded} if $X$ can be written as the countable union of functionally bounded sets in $X$.

\begin{proposition}\label{p:sb} If a uniform space $X$ is $\sigma$-bounded, then its function space $C_u(X)$ is $\w^\w$-dominated in $\IR^X$.
\end{proposition}

\begin{proof} Write $X$ as the countable union $X=\bigcup_{n\in\w}B_n$ of functionally bounded sets $B_n$. For every $\alpha\in\w^\w$ consider the function $f_\alpha:X\to\IR$ such that $f_\alpha(x)=\alpha(n)$ for any $n\in\w$ and $x\in B_n\setminus\bigcup_{k<n}B_k$. It is easy to see that the monotone correspondence $\w^\w\to\IR^X$, $\alpha\mapsto f_\alpha$, witnesses that the set  $C_u(X)$ is $\w^\w$-dominated in $\IR^X$.
\end{proof}

For Tychonoff spaces (endowed with their universal uniformity) we can prove a bit more.

\begin{proposition}\label{p:ACFK} For a Tychonoff space $X$ we have the implications
$(1)\Ra(2)\Leftrightarrow(3)\Ra(4)$ of the following properties:
\begin{enumerate}
\item $X$ is $\sigma$-bounded;
\item the set $C(X)$ is $\w^\w$-dominated in $\IR^X$;
\item $C_p(X)$ is contained in a $K$-analytic subspace of $\IR^X$;
\item each  metrizable image of $X$ is $\sigma$-compact.
\end{enumerate}
\end{proposition}

\begin{proof} The implication $(1)\Ra(2)$ was proved in Proposition~\ref{p:sb}, the equivalence $(2)\Leftrightarrow(3)$ follow from Proposition 9.6 of \cite{kak} and the implication $(2)\Ra(4)$ is proved in \cite[7.1.1]{Ban}.
\end{proof}

\begin{remark}  By Example 2 of \cite{Lei}, there exists a Tychonoff space with a unique non-isolated point $X$ such that
$X$ is Lindel\"of, the function space $C_p(X)$ is $K$-analytic but $X$ is not $\sigma$-bounded. This example shows that the implication $(1)\Ra(3)$ in Proposition~\ref{p:ACFK} cannot be reversed, which answers a problem posed by Arhangel'ski and Calbrix in \cite{AC} and then repeated in \cite[p.~215]{kak}.
\end{remark}

\section{Free locally convex spaces over uniform spaces}\label{s:Lc}

In this section we study free locally convex spaces of uniform spaces.
For a uniform space $X$ its {\em free locally convex space} is a pair $(\Lc_u(X),\delta_X)$ consisting of a locally convex space $\Lc_u(X)$ and a uniformly continuous map $\delta_X:X\to \Lc_u(X)$ such that for every uniformly continuous map $f:X\to Y$ into a locally convex space $Y$ there exists a continuous linear operator $\bar f:\Lc_u(X)\to Y$ such that $\bar f\circ\delta_X=f$. By a standard technique
(see \cite{Rai}), it can be shown that for every uniform space $X$ the free locally convex space $(\Lc_u(X),\delta_X)$ exists and is unique up to a topological isomorphism.

For a topological space $X$ its free locally convex space $(\Lc(X),\delta_X)$ coincides with the free locally convex space $(\Lc_u(X),\delta_X)$ of the space $X$ endowed with the universal uniformity $\U_X$.

For a uniform space $X$ by $C_u(X)$ we denote the linear space of all uniformly continuous real-valued functions on $X$. A subset $E\subset C_u(X)$ is called
\begin{itemize}
\item {\em pointwise bounded} if for every point $x\in X$ the set $E(x)=\{f(x):f\in E\}$ is bounded in the real line;
\item {\em equicontinuous} if for every $\e>0$ there is an entourage $U\in\U_X$ such that $|f(x)-f(y)|<\e$ for all $f\in E$ and $(x,y)\in U$.
\end{itemize}

By $\E(C_u(X))$ we denote the set of pointwise bounded equicontinuous subfamilies in $C_u(X)$.
The set $\E(C_u(X))$ is a poset with respect to the natural inclusion order (defined by $E\le F$ iff $E\subset F$).

Let $C^*_u(X)$ be the linear space of all linear functionals on $C_u(X)$, endowed with the topology of uniform convergence on pointwise bounded equicontinuous subsets of $C_u(X)$. A neighborhood base of this topology at zero consists of the polar sets
$${E}^*=\{\mu\in C^*_u(X):\sup_{f\in E}|\mu(f)|\le1\}\mbox{ \ where $E\in\E(C_u(X))$}.$$

Each element $x\in X$ can be identified with the Dirac measure $\delta_x\in C_u^*(X)$ assigning to each function $f\in C_u(X)$ its value $f(x)$ at $x$. It can be shown that the map $\delta:X\to C^*_u(X)$, $\delta:x\mapsto\delta_x$, is uniformly continuous in the sense that for any neighborhood $V\subset C_u^*(X)$ of zero there is an entourage $U\in\U_X$ such that $\delta_x-\delta_y\in V$ for any $(x,y)\in U$.

Let $\Lc_u(X)$ be the linear hull of the set $\delta(X)=\{\delta_x:x\in X\}$ in $C_u^*(X)$. By \cite{Rai}, the pair $(\Lc_u(X),\delta)$ is a free locally convex space over the uniform space $X$.

In the space $\Lc_u(X)$ consider the subspaces
$$
\begin{aligned}
&\IN(X{-}X)=\{n(\delta_x-\delta_y):n\in\IN,\;x,y\in X\},\\
&\tfrac1\IN X=\{\tfrac1n\delta_x:n\in\IN,\;x\in X\}, \mbox{ and}\\
&\Lambda_u(X)= \IN(X{-}X)\cup\tfrac1\IN X.
\end{aligned}
$$

\begin{lemma}\label{l:u1} For any uniform space $X$ we have the reductions
$$\E(C_u(X))\succcurlyeq \Tau_0(C^*_u(X))\succcurlyeq \Tau_0(\Lc_u(X))\succcurlyeq \Tau_0(\Lambda_u(X))\succcurlyeq\E(C_u(X)),$$
therefore all posets above are equivalent each other.
\end{lemma}

\begin{proof} The monotone cofinal map $\E(C_u(X))\to \Tau_0(C_u^*(X))$,
which assigns to $E$ the polar set  ${E}^*$,
 witnesses that $\E(C_u(X))\succcurlyeq \Tau_0(C^*_u(X))$. The reductions $\Tau_0(C^*_u(X))\succcurlyeq \Tau_0(\Lc_u(X))\succcurlyeq \Tau_0(\Lambda_u(X))$ trivially follow from the inclusions $\Lambda_u(X)\subset \Lc_u(X)\subset C_u^*(X)$.

To prove that  $\Tau_0(\Lambda_u(X))\succcurlyeq\E(C_u(X))$, for every neighborhood $V\subset \Lambda_u(X)$ of zero, consider the set
$$E_V=\bigcap_{\mu\in V}\{f\in C_u(X):|\mu(f)|\le 1\}.$$
First we show that the set $E_V$ is pointwise bounded and equicontinuous.

To see that $E_V$ is pointwise bounded, fix any point $x\in X$. Taking into account that the sequence $\big(\frac1n\delta_x\big)_{n\in\IN}$ converges to zero in $\Lambda_u(X)$, find a number $n\in\IN$ with $\frac1n\delta_x\in V$. Then for every $f\in E_V$ we get $|\frac1n\delta_x(f)|\le 1$ and hence $|f(x)|\le n$, which means that the set $E_V$ is pointwise bounded.

Next, we check that $E_V$ is equicontinuous. Given any point $x\in X$ and $\e>0$ we should find an entourage $U\in\U_X$ such that $|f(x)-f(y)|<\e$ for all $f\in E_V$ and $(x,y)\in U$.
Find $n\in\IN$ such that $\frac1n<\e$. The uniform continuity of the map $n\delta:X\to C^*_u(X)$, $n\delta:x\mapsto n\cdot\delta_x$,  yields an entourage $U\in\U_X$ such that $n(\delta_x-\delta_y)\in V$ for any $(x,y)\in U$. Then for any $f\in E_V$ the inclusion $n(\delta_x-\delta_y)\in V$ implies $|n(\delta_x(f)-\delta_y(f))|\le 1$ and  $|f(x)-f(y)|\le \frac1n<\e$, which means that the set $E_V$ is equicontinuous.

So, the map $\Tau_0(\Lambda_u(X))\to \E(C_u(X))$, $V\mapsto E_V$, is well-defined. It is clear that this map is monotone. It remains to show that it is cofinal. Given any pointwise bounded equicontinuous set $E\in\E(C_u(X))$, consider the neighborhood $V=\{\mu\in \Lambda_u(X):\sup_{f\in E}|\mu(f)|\le 1\}$ of zero in $\Lambda_u(X)$ and observe that $E\subset E_V$. So, $\Tau_0(\Lambda_u(X))\succcurlyeq \E(C_u^*(X))$.
\end{proof}

As before, in the following lemma we consider the space $C_u(X)$ as a poset endowed with the partial order $f\le g$ iff $f(x)\le g(x)$ for all $x\in X$.

\begin{lemma}\label{l:u2} For any uniform space $X$ of density $\kappa$ we have the reductions $$\U_X^\w\times \w^\kappa\succcurlyeq\E(C_u(X))\cong\U_X^\w\times C_u(X).$$
\end{lemma}

\begin{proof} Fix a dense subset $D\subset X$ of cardinality $\kappa$. To see that $\U_X^\w\times \w^D\succcurlyeq \E(C_u(X))$, for any pair $P=\big((U_n)_{n\in\w},\varphi)\in \U_X^\w\times \w^D$ consider the pointwise bounded equicontinuous set
$$E_P=\bigcap_{x\in D}\{f\in C_u(X)\colon |f(x)|\le\varphi(x)\}\cap\bigcap_{n\in\w}\bigcap_{x,y\in U_n}\{f\in C_u(X)\colon |f(x)-f(y)|\le \tfrac1{2^{n}}\}.$$
It is clear that the map $\U_X^\w\times \w^D\to \E(C_u(X))$, $P\mapsto E_P$, is monotone and cofinal, witnessing that $\U_X^\w\times \w^D\succcurlyeq \E(C_u(X))$.
By analogy we can prove that $\U_X^\w\times C_u(X)\succcurlyeq \E(C_u(X))$.

To see that $\E(C_u(X))\succcurlyeq \U_X^\w\times C_u(X)$, to every pointwise bounded equicontinuous set $E\in \E(C_u(X))$ assign the function $$\varphi_E\in C_u(X),\;\;\varphi_E:x\mapsto\sup_{f\in E}|f(x)|,$$ and the sequence $(E_n)_{n\in\w}$ of entourages
$$E_n:=\bigcap_{f\in E}\{(x,y)\in X\times X:|f(x)-f(y)|\le2^{-n}\}\in\U_X,\;\;n\in\w.$$
It is clear that the map $\E(C_u(X))\succcurlyeq \U(X)^\w\times C_u(X)$, $E\mapsto \big((E_n)_{n\in\w},\varphi_E\big)$, is well-defined and monotone. It remains to prove that this map is cofinal.

Fix any pair $P=\big((U_n)_{n\in\w},\varphi\big)\in \U_X^\w\times C_u(X)$. Let $V_{-1}=X\times X$ and choose a sequence of entourages $(V_n)_{n\in\w}\in\U_X^\w$ such that $V_n=V_n^{-1}$ and $V_n\circ V_n\circ V_n\subset U_n\cap V_{n-1}$ for all $n\in\w$. By Theorem~8.1.10 \cite{Eng}, there is a pseudometric $d:X\times X\to[0,1]$ such that
$$\{(x,y)\in X\times X:d(x,y)<2^{-n}\}\subset V_n\subset \{(x,y)\in X\times X:d(x,y)\le 2^{-n}\}$$for all $n\in\w$.
For every $z\in X$ consider the uniformly continuous function $f_z\in C_u(X)$ defined by $f_z(x)=\varphi(x)+d(z,x)$ for $x\in X$. It is easy to see that the function family $E=\{f_z:z\in X\}$ is pointwise bounded and equicontinuous. We claim that $((U_n)_{n\in\w},\varphi)\le ((E_n)_{n\in\w},\varphi_E)$. It is clear that $\varphi\le\varphi_E\le \varphi+1$. Next, we show that for every $n\in\w$ the set $E_n=\bigcap_{f\in E}\{(x,y)\in X\times X:|f(x)-f(y)|<2^{-n}\}\subset U_n$. Indeed, for any $(x,y)\in E_n$ we get $d(x,y)=|f_x(x)-f_x(y)|<2^{-n}$ and hence $(x,y)\in V_n\subset U_n$ by the choice of the pseudometric $d$ and the entourage $V_n$.
\end{proof}

Lemmas~\ref{l:u1} and \ref{l:u2} imply the following reduction theorem.

\begin{theorem}\label{t:u} For any uniform space $X$ of density $\kappa$ we have the reductions:
$$\Tau_0(C_u^*(X))\cong \Tau_0(\Lc_u(X))\cong \Tau_0(\Lambda_u(X))\cong \E(C_u(X))\cong \U_X^\w\times C_u(X)\preccurlyeq \U_X^\w\times \w^\kappa.$$
\end{theorem}

As an immediate consequence, we obtain a key characterization of uniform spaces $X$ whose free locally convex spaces $\Lc_u(X)$ have local  $\w^\w$-bases.

\begin{corollary}\label{c:Lu3} The free locally convex space $\Lc_u(X)$ of a uniform space $X$ has an  $\w^\w$-base if and only if the uniformity $\U_X$ of $X$ has an $\w^\w$-base and the poset $C_u(X)$ is $\w^\w$-dominated.
\end{corollary}

Combining this corollary with Proposition~\ref{p:sb}, we get another

\begin{corollary}\label{c:Lu4} The free locally convex space $\Lc_u(X)$ of a $\sigma$-bounded uniform space $X$ has an  $\w^\w$-base if and only if the uniformity $\U_X$ of $X$ has an $\w^\w$-base.
\end{corollary}

\begin{example}\label{ex:Lc-conv} For any convex subset $X$ of a Banach space the free locally convex space $\Lc_u(X)$  has an $\w^\w$-base.
\end{example}

\section{Free topological vector spaces over uniform spaces}\label{s:Lin}

For a uniform space $X$ its {\em free topological vector space} is a pair $(\Lin_u(X),\delta_X)$ consisting of a topological vector space $\Lin_u(X)$ and a uniformly continuous map $\delta_X:X\to \Lin_u(X)$ such that for every uniformly continuous map $f:X\to Y$ into a  topological vector space $Y$ there exists a continuous linear operator $\bar f:\Lin_u(X)\to Y$ such that $\bar f\circ\delta_X=f$.
For a Tychonoff space $X$ its free topological vector space $(\Lin(X),\delta_X)$ coincides with the free  topological vector space $(\Lin_u(X),\delta_X)$ of the space $X$ endowed with the universal uniformity $\U_X$.

Given a uniform space $X$, we shall identify the free topological vector space $\Lin_u(X)$ of $X$ with the linear hull $L(X)$ of the set $\delta(X)=\{\delta_x:x\in X\}\subset C^*_u(X)$, endowed with the strongest topology turning $L(X)$ into a topological vector space and making the map $\delta:X\to L(X)$, $\delta:x\mapsto\delta_x$, uniformly continuous.

The topology of the space $\Lin_u(X)$ can be described as follows.
For a function $\varphi(x) \in \IR^X$ denote by $\hat{\varphi}:X\to[1,\infty)$ the function defined by $\hat\varphi(x) = \max\{1,\varphi(x)\}$ for $x\in X$.
 Given a pair  $\big((U_n)_{n\in\w},(\varphi_n)_{n\in\w}\big)\in\U_X^\w\times (\IR^X)^\w$, consider the sets
$$
\begin{aligned}
&\sum_{n\in\w}U_n=\bigcup_{m\in\w}\Big\{\sum_{n=0}^mt_n(\delta_{x_n}-\delta_{y_n}):|t_n|\le 1,\;(x_n,y_n)\in U_n\mbox{ for all }n\le m\Big\}\mbox{ and }\\
&\sum_{n\in\w}\tfrac1{\varphi_n} X=\bigcup_{m\in\w}\Big\{\sum_{n=0}^mt_n\delta_{x_n}:x_n\in X,\;
|t_n|\le
1/\hat{\varphi}_n(x)
\Big\}.
%\mbox{ \ and}\\
%&V_P=\sum_{n\in\w}U_n+\sum_{n\in\w}\tfrac1{\varphi_n} X.
\end{aligned}
$$

We recall that for a uniform space $X$ by $C_\w(X)$ we denote the subset of $\IR^X$ consisting of $\w$-continuous functions $f:X\to\IR$.
A subset $D\subset \IR^X$ is called {\em directed} if for any functions $f,g\in D$ there is a function $h\in D$ such that $h\ge\max\{f,g\}$. We shall say that a set $D\subset \IR^X$ {\em dominates} $C_\w(X)$ if for every function $f\in C_\w(X)$ there exists a function $g\in D$ with $g\ge f$.

\begin{theorem}\label{t:lin-top} Let $X$ be a uniform space and $D\subset\IR^X$ be a directed subset dominating the set $C_\w(X)$ of $\w$-continuous functions on $X$. The family
$$\mathcal B=\Big\{\sum_{n\in\w}U_n+\sum_{n\in\w}\tfrac1{\varphi_n} X:(\U_n)_{n\in\w}\in\U_X^\w,\;(\varphi_n)_{n\in\w}\in D^\w\Big\}$$is a neighborhood base at zero of the topology of the free topological vector space $\Lin_u(X)$. Consequently, $$\U_X^\w\times D^\w\succcurlyeq \Tau_0(\Lin_u(X)).$$
\end{theorem}

\begin{proof} Given a pair $P=\big((U_n)_{n\in\w},(\varphi_n)_{n\in\w}\big)\in\U_X^\w\times D^\w$, denote the set $$\sum_{n\in\w}U_n+\sum_{n\in\w}\tfrac1{\varphi_n} X$$ by $V_P$. Let $\tau$ be the topology on $L(X)$ consisting of the sets $W\subset L(X)$ such that for any $w\in W$ there is a pair $P\in\U_X^\w\times D^\w$ such that $w+V_P\subset W$. The definition of the set $V_P$ implies that it belongs to the topology $\tau$.

We claim that the topology $\tau$ turns the vector space $L(X)$ into a topological vector space. We need to check the continuity of the addition $L(X)\times L(X)\to L(X)$, $(u,v)\mapsto u+v$, and the multiplication $\IR\times L(X)\to L(X)$, $(t,u)\mapsto tu$, with respect to the topology $\tau$. Since the topology $\tau$ is invariant under shifts, it suffices to check the continuity of the addition at zero. Given a neighborhood $W\in\tau$ of zero, find a pair $P=\big((U_n)_{n\in\w},(\varphi_n)_{n\in\w}\big)\in\U_X^\w\times D^\w$ such that such that $V_P\subset W$. For every $n\in\w$ consider the entourage $U'_n=U_{2n}\cap U_{2n+1}\in \U_X$ and choose a function $\psi_n\in D$ such that $\psi_n\ge \max\{\varphi_{2n},\varphi_{2n+1}\}$. We claim that for the pair $Q=\big((U_n')_{n\in\w},(\psi_n)_{n\in\w}\big)$ we get $V_Q+V_Q\subset V_P$. Indeed,
\begin{multline*}
V_Q+V_Q=\sum_{n\in\w}\tfrac1{\psi_n} X+\sum_{n\in\w}U_n'+\sum_{n\in\w}\tfrac1{\psi_n} X+\sum_{n\in\w}U_n'\subset\\
\subset\sum_{n\in\w}\tfrac1{\varphi_{2n}}X+\sum_{n\in\w}\tfrac1{\varphi_{2n+1}}X
+\sum_{n\in\w}U_{2n}+\sum_{n\in\w}U_{2n+1}=\sum_{n\in\w}\tfrac1{\varphi_n} X+\sum_{n\in\w}U_n=V_P.
\end{multline*}
The definition of the set $V_Q$ implies that it is open in the topology $\tau$. So, $V_Q$ is an open neighborhood of zero such that $V_Q+V_Q\subset V_P\subset W$, which proves the continuity of the addition at zero.

Next, we prove the continuity of the multiplication by a scalar.
%Since the topology $\tau$ is invariant, it suffices to check that the multiplication map $\cdot:\IR\times LH(X)\to LH(X)$ is separately continuous and continuous at zero. The continuity of the multiplication at zero follows from the obvious inclusion $[-1,1]\cdot V_P\subset V_P$ for all $P\in\U(X)^\w\times C_u(X)^\w$. To see that the multiplication map is separately continuous,
Fix any pair $(t,u)\in\IR\times L(X)$ and any neighborhood $W\in\tau$ of the product $tu$.  Find $m\in\w$ such that $|t|\le m$ and $u=\sum_{n=0}^mt_n\delta_{x_n}$ for some real numbers $t_0,\dots,t_n$ and some points $x_0,\dots,x_n\in X$. By the continuity of the addition at zero, there exists a neighborhood $W_0\in\tau$ of zero such that the set $\sum_{n=0}^{m+1}W_0=\big\{\sum_{n=0}^{m+1}w_n:w_0,\dots,w_{m+1}\in W_0\big\}$ is contained in the neighborhood $W-tu$ of zero. Find a pair $Q=\big((U_n')_{n\in\w},(\psi_n)_{n\in\w}\big)\in\U_X^\w\times D^\w$ such that $V_Q\subset W_0$. Choose $\e\in(0,1]$ so small that $\e t_n\le 1/\hat{\psi}_n(x_n)$ for all $n\le m$, which implies that $[-\e,\e]\cdot u\subset V_Q$. We claim that $(t-\e,t+\e)\cdot (u+V_Q)\subset W$. Indeed, take any pair $(t',u')\in(t-\e,t+\e)\times (u+V_Q)$ and observe that
\begin{multline*}
t'u'=tu+t(u'-u)+(t'-t)u+(t'-t)(u'-u)\in\\
\in tu+tV_Q+[{-}\e,\e]u+[-1,1]V_Q\subset tu+ mV_Q+V_Q+V_Q\subset tu+\sum_{n=0}^{m+1}W_0\subset W.
\end{multline*}

Therefore the topology $\tau$ turns the linear hull $L(X)$ of the set $\{\delta_x:x\in X\}$ into a topological vector space. The definition of the topology $\tau$ implies that the map $\delta:X\to L(X)$, $\delta:x\mapsto\delta_x$, is uniformly continuous. Now the definition of the free  topological vector space $\Lin_u(X)$ guarantees that the identity map $\Lin_u(X)\to L(X)$ is continuous (with respect to the topology $\tau$). We claim that this map is a homeomorphism.

Given any neighborhood $W\subset \Lin_u(X)$ of zero, we should find a pair $P\in\U_X^\w\times D^\w$ such that $V_P\subset W$.
Consider the constant map $\mathbf 1:X\to\{1\}\subset\IR$ and the induced linear continuous operator $\bar{\mathbf 1}:\Lin_u(X)\to\IR$. Replacing $W$ by a smaller neighborhood of zero, we  can assume that $\bar{\mathbf 1}(W)\subset(-1,1)$. By an argument similar to that of \cite[\S 1.2]{Rolewicz}, we can find an invariant continuous pseudometric $d$ on $\Lin_u(X)$ such that $\{x\in \Lin_u(X):d(x,0)<1\}\subset W$ and $d(tx,0)\le d(x,0)$ for any $t\in[-1,1]$ and $x\in \Lin_u(X)$. For every $n\in\w$ consider the entourage $U_n=\{(x,y)\in X\times X:d(\delta_x,\delta_y)<2^{-n-2}\}$. Using the paracompactness of the pseudometric space $(\Lin_u(X),d)$, for every $n\in\w$ it is easy to construct a $d$-continuous function $\psi_n:\Lin_u(X)\to[1,\infty)$ such that $d(tx,0)<2^{-n-2}$ for every $x\in \Lin_u(X)$ and $t\in\IR$ with $|t|\le 1/\psi_n(x)$. The uniform continuity of the map $\delta:X\to \Lin_u(X)$ and the $d$-continuity of the map $\psi_n$ implies the $\w$-continuity of the composition $\psi_n\circ\delta:X\to\IR$. Finally, choose a function $\varphi_n\in D$ such that $\varphi_n\ge \psi_n\circ\delta$.
It follows that the pair $P=\big((U_n)_{n\in\w},(\varphi_n)_{n\in\w}\big)$ belongs to the poset $\U_X^\w\times D^\w$. Using the triangle inequality it is easy to check that $V_P\subset W$.

Therefore the correspondence $\U_X^\w\times D^\w\to \Tau_0(\Lin_u(X))$, $P\mapsto V_P$, is monotone and cofinal, yielding the reduction $\U_X^\w\times D^\w\succcurlyeq \Tau_0(\Lin_u(X))$.
\end{proof}

\begin{theorem}\label{t:Lin+Lc-reduct} For any uniform space $X$ and any directed subset $D\subset \IR^X$ dominating $C_\w(X)$, we get the reductions
$$\U_X^\w\times D^\w\succcurlyeq \Tau_0(\Lin_u(X))\succcurlyeq \Tau_0(\Lc_u(X))\cong \U_X^\w\times C_u(X).$$
\end{theorem}

\begin{proof} The reduction $\U_X^\w\times D^\w\succcurlyeq \Tau_0(\Lin_u(X))$ was proved in Theorem~\ref{t:lin-top}.
To see that $\Tau_0(\Lin_u(X))\succcurlyeq \Tau_0(\Lc_u(X))$, consider the monotone map $\Tau_0(\Lin_u(X))\to\Tau_0(\Lc_u(X))$ assigning to each neighborhood $U\in \Tau_0(\Lin_u(X))$ of zero its convex hull $\conv(U)$. This map is clearly well-defined.
Its cofinality follows from \cite[Proposition 5, TVS II.27]{Bourbaki}.
 The final reduction  $\Tau_0(\Lc_u(X))\cong \U_X^\w\times C_u(X)$ was proved in Theorem~\ref{t:u}.
 \end{proof}

\begin{corollary}\label{c:Lin-dom} Let $X$ be a uniform space
whose uniformity has an $\w^\w$-base. If the set $C_\w(X)$ is $\w^\w$-dominated in $\IR^X$, then the free topological vector space $\Lin_u(X)$ has a local $\w^\w$-base.
\end{corollary}

\begin{proof} If the set $C_\w(X)$ is $\w^\w$-dominated in $\IR^X$, then we can find a set $D=\{f_\alpha\}_{\alpha\in\w^\w}\subset\IR^X$ which dominates $C_\w(X)$ and such that $f_\alpha\le f_\beta$ for all $\alpha\le\beta$ in $\w^\w$. It follows that $D$ is a directed set with $\w^\w\succcurlyeq D$. Applying Theorem~\ref{t:Lin+Lc-reduct}, we get the reductions
$$\w^\w\cong (\w^\w)^\w\times(\w^\w)^\w\succcurlyeq D^\w\times \U_X^\w\succcurlyeq\Tau_0(\Lin_u(X)),$$witnessing that the space $\Lin_u(X)$ has a local $\w^\w$-base.
\end{proof}

\begin{corollary}\label{c:Lin-subs} Let $X$ be a uniform space
whose uniformity has an $\w^\w$-base. If for some dense subspace $Z\subset X$ the set $C_\w(Z)$ is $\w^\w$-dominated in $\IR^Z$, then the free topological vector space $\Lin_u(X)$ has an $\w^\w$-base.
\end{corollary}

\begin{proof} By Corollary~\ref{c:Lin-dom}, the free topological vector space $\Lin_u(Z)$ has an $\w^\w$-base and so does its completion $\bar \Lin_u(Z)$. For a uniform space $Y$ by $c_Y:\Lin_u(Y)\to\bar \Lin_u(Y)$ we denote the identity embedding of the free topological vector space $\Lin_u(Y)$ into its completion. Let $i:Z\to X$ be the identity embedding. By the functoriality of the constructions $\Lin_u$ and $\bar\Lin_u$, we obtain a commutative diagram:
$$
\xymatrix{
Z\ar^{\delta_Z}[r]\ar_i[d]&\Lin_u(Z)\ar^{c_Z}[r]\ar^{\Lin_u i}[d]&\bar\Lin_u(Z)\ar^{\bar\Lin_u i}[d]\\
X\ar_{\delta_X}[r]&\Lin_u(X)\ar_{c_X}[r]&\bar\Lin_u(X).
}
$$
We claim that the linear operator $\bar V_ui:\bar V_u(Z)\to\bar V_u(X)$ is a topological isomorphism.
For this observe that the uniformly continuous map $c_Z\circ\delta_Z:Z\to\bar \Lin_u(Z)$ admits a unique uniformly continuous extension $e:X\to \bar\Lin_u(Z)$, which induces a unique linear continuous operator $\bar e:\bar \Lin_u(X)\to\bar\Lin_u(Z)$ such that $\bar e\circ c_X\circ\delta_X=e$.  It follows that $\bar e\circ \bar\Lin_u i$ is the identity map on the set $c_Z\circ \delta_Z(Z)\subset \bar V_u(Z)$ and also on its  linear hull which is dense in $\bar\Lin_u(Z)$. Now the continuity of the composition $\bar e\circ\bar \Lin_u i$ guarantees that $\bar e\circ\bar\Lin_u i$ is the identity map of $\bar\Lin_u(Z)$. By analogy we can prove that $\bar \Lin_u i\circ\bar e$ is the identity map of $\bar\Lin_u(X)$. Consequently, the topological vector spaces $\bar\Lin_u(Z)$ and $\bar \Lin_u(X)$ are topologically isomorphic, which implies that the space $\bar \Lin_u(X)$ has an $\w^\w$-base and so does its subspace $\Lin_u(X)$.
\end{proof}

A subset $D$ of a uniform space $X$ is called {\em uniformly discrete} if there exists an entourage $U\in\U$ such that $(x,y)\notin U$ for any distinct points $x,y\in D$. Observe that a uniform space $X$ is $\w$-narrow if and only if each uniformly discrete subset of $X$ is at most countable.

\begin{theorem}\label{t:Lin-narrow} For any uniformly discrete subset $D$ of a uniform space $X$ we have the reduction $\Tau_0(\Lin_u(X))\succcurlyeq\w^D$.
\end{theorem}

\begin{proof} The reduction $\Tau_0(\Lin_u(X))\succcurlyeq\w^D$ is defined by the monotone function $\varphi_*:\Tau_0(\Lin_u(X))\to\w^D$, $\varphi_*:V\mapsto \varphi_V$, where $\varphi_V:D\to\w$, $\varphi_V:x\mapsto\min\{n\in\w:\frac1{2^{n}}\delta_x\in V\}$. Let us show that the function $\varphi_*$ is cofinal. Given any function $\psi\in \w^D$, we should find a neighborhood $V\in\Tau_0(\Lin_u(X))$ such that $\varphi_V\ge\psi$. Taking into account that the set $D$ is uniformly discrete, we can apply Theorem 8.1.10 \cite{Eng} and find a uniform pseudometric $d:X\times X\to [0,3]$ on $X$ such that $d(x,y)=3$ for any distinct points $x,y\in D$. For every $z\in D$ consider the function
$\Lambda_z:X\to[0,2]$, $\Lambda_z:x\mapsto 2\max\{0,1-d(x,z)\}$. Observe that the function $\Psi:X\to\IR$, $\Psi:x\mapsto\sum_{z\in D}\psi(z)\Lambda_z(x)$ is $\w$-continuous.

For every $n\in\w$ consider the entourage $U_n=\{(x,y)\in X\times X:d(x,y)<\tfrac1{2^{n+2}}\}\in\U_X$ and the $\w$-continuous function $\Psi_n=2^{n+1+\Psi}\ge1$. By Theorem~\ref{t:lin-top}, the set   $$V=\sum_{n\in\w}U_n+\sum_{n\in\w}\tfrac1{\Psi_n}X$$is a neighborhood of zero in $\Lin_u(X)$. We claim that $\varphi_V\ge\psi$. To derive a contradiction, assume that $\varphi_V\not\ge\psi$ and find a point $x\in D$ with $\varphi_V(x)<\psi(x)$. The definition of the number $k=\varphi_V(x)<\psi(x)$ guarantees that $\frac1{2^{k}}\delta_x\in V$ and hence $\frac1{2^{k}}\delta_x=\sum_{i\in m}a_i(\delta_{x_i}-\delta_{y_i})+\sum_{i\in m} b_i\delta_{z_i}$ for some $m\in\w$ and some sequences $(a_i)_{i\in m}\in[-1,1]^m$, $\big((x_i,y_i)\big)_{i\in m}\in \prod_{i\in m} U_i$, $(z_i)_{i\in m}\in X^m$ and $(b_i)_{i\in m}\in [-1,1]^m$ with $|b_i|\le \frac1{\hat\Psi_i(z_i)}=\frac1{\Psi_i(z_i)}$ for all $i\in m$.

Consider the graph $\Gamma=(F,E)$ whose set of vertices is the finite set $F=\{x\}\cup\{x_i\}_{i\in m}\cup\{y_i\}_{i\in m}\cup \{z_i\}_{i\in m}$ and two vertices $u,v\in F$ are linked by an edge if and only if $\{u,v\}=\{x_i,y_i\}$ for some $i\in m$. In this case $d(u,v)=d(x_i,y_i)<\frac1{2^{i+2}}$. It follows that connected components of the graph $\Gamma$ have diameter at most $\sum_{i\in m}\frac1{2^{i+2}}<\frac12$. In particular, the connected component $C_x\subset F$ of the point $x$ has diameter $<\frac12$. Observe that for every $i\in m$ with $z_i\in C_x$ we get $d(x,z_i)<\frac12$ and hence $\Psi(z_i)=\psi(x)\Lambda_x(z_i)=2\psi(x)(1-d(x,z_i))\ge\psi(x)$.

Consider the characteristic function $\chi:F\to\{0,1\}\subset\IR$ of the set $C_x$ (which means that $\chi^{-1}(1)=C_x$) and observe that $\chi$ is constant on each connected component of the graph $\Gamma$. Let $\bar\chi:\Lin_u(F)\to \IR$ be the unique linear operator such that $\bar \chi\circ\delta_X=\chi$. It follows that
$$
\begin{aligned}
\frac1{2^{k}}&=\bar \chi\big(\tfrac1{2^{k}}\delta_x\big)=\bar\chi\Big(\sum_{i\in m}a_i(\delta_{x_i}-\delta_{y_i})+\sum_{i\in m}^m b_i\delta_{z_i}\Big)=\sum_{i\in m}a_i(\chi(x_i)-\chi(y_i))+\sum_{i\in m} b_i\chi(z_i)=\\
&=\sum_{i\in m}b_i\chi(z_i)=\sum_{z_i\in C_x}b_i\le\sum_{z_i\in C_x}\frac1{\Psi_i(z_i)}=\sum_{z_i\in C_x}\frac1{2^{i+1+\Psi(z_i)}}\le\sum_{z_i\in C_x}\frac1{2^{i+1+\psi(x)}}<\frac1{2^{\psi(x)}}
\end{aligned}
$$
and hence $\varphi_V(x)=k>\psi(x)$, which contradicts our assumption. This contradiction shows that $\varphi_V\ge\psi$, which means that the monotone map $\varphi_*:\Tau_0(\Lin_u(X))\to\w^D$ is cofinal and yields the reduction $\Tau_0(\Lin_u(X))\succcurlyeq\w^D$.
\end{proof}

\begin{corollary}\label{c:Lin-nar} If the free topological vector space $\Lin_u(X)$ of a uniform space $X$ has an $\w^\w$-base, then the uniform space $X$ is $\w$-narrow and $\w^\w$-based. Moreover, under $\w_1<\mathfrak b$, the space $X$ is separable.
\end{corollary}

\begin{proof} Assuming that $X$ is not $\w$-narrow, we can find a uniformly discrete subset $D\subset X$ of cardinality $|D|=\w_1$. By Theorem~\ref{t:Lin-narrow}, $\w^\w\succcurlyeq\Tau_0(\Lin_u(X))\succcurlyeq\w^D\cong\w^{\w_1}$, which contradicts Proposition~\ref{p:e(X)}(2). So, the space $X$ is $\w$-narrow. Theorem~\ref{t:Lin+Lc-reduct} yields the reductions $\Tau_0(\Lin_u(X))\succcurlyeq\Tau_0(\Lc_u(X))\succcurlyeq\U_X$ implying that the uniform space $X$ is $\w^\w$-based.

By Theorem~\ref{t:looong}, under $\w_1<\mathfrak b$ the $\w$-narrow $\w^\w$-based uniform space $X$ is separable.
\end{proof}

\section{Free (locally convex) topological vector spaces}

In this section we unify the results of Sections~\ref{s:Lc} and \ref{s:Lin} and prove the following chracterization.

\begin{theorem}\label{t:Lin+LcU} For a uniform space $X$ consider the following conditions:
\begin{itemize}
\item[$(\Lin)$] the free topological vector space $\Lin_u(X)$ has an $\w^\w$-base;
\item[$(\Lc)$] the free locally convex space $\Lc_u(X)$ has an $\w^\w$-base;
\item[$(\U)$] the uniformity $\U_X$ of $X$ has an $\w^\w$-base;
\item[$(C)$] the poset $C_u(X)$ is $\w^\w$-dominated;
\item[$(s)$] the space $X$ is separable.
\item[$(\sigma)$] the space $X$ is $\sigma$-compact.
\end{itemize}
Then $(\U{\wedge}\sigma)\Ra(\U{\wedge}s)\Ra(\Lin)\Ra(\Lc)\Leftrightarrow(\U{\wedge}C)$ and moreover $(\U{\wedge}s)\Leftrightarrow(\Lin)$ under $\w_1<\mathfrak b$.\newline
If the completion of $X$ is $\IR$-complete, then $(\U{\wedge}s)\Leftrightarrow (\Lin)\Leftrightarrow(\Lc)\Leftrightarrow(\U{\wedge}C)$.
\newline If the uniform space $X$ is $\IR$-universal, then $(\U{\wedge}\sigma)\Leftrightarrow(\U{\wedge}s)\Leftrightarrow (\Lin)\Leftrightarrow(\Lc)\Leftrightarrow(\U{\wedge}C)$.
\end{theorem}

\begin{proof} The implications $(\Lin)\Ra(\Lc)\Leftrightarrow(\U{\wedge}C)$ follow from the reductions
$\Tau_0(\Lin_u(X))\succcurlyeq \Tau_0(\Lc_u(X))\cong \U_X^\w\times C_u(X)$ proved in Theorem~\ref{t:Lin+Lc-reduct}. The implication $(\U{\wedge}\sigma)\Ra(\U{\wedge}s)$ follows from the metrizability of compact subsets of $\w^\w$-based uniform spaces (proved by Cascales and Orihuela \cite{CO}, see also Theorem~\ref{t:Sigma}). To prove that $(\U{\wedge}s)\Ra(\Lin)$, assume that the space $X$ is separable and the uniformity $\U_X$ of $X$ has an $\w^\w$-base. Fix a countable dense subspace $Z\subset X$ and observe that the set $C_\w(Z)$ is $\w^\w$-dominated in $\IR^Z$ (because $\IR^Z\cong \IR^\w\cong\w^\w$). By Corollary~\ref{c:Lin-subs}, the free topological vector space $\Lin_u(X)$ has an $\w^\w$-base, which completes the proof of the implications $(\U{\wedge}s)\Ra(\Lin)\Ra(\Lc)\Leftrightarrow(\U{\wedge}C)$. Under $\w_1<\mathfrak b$ the implication $(\U{\wedge}s)\Leftarrow(\Lin)$ was proved in Corollary~\ref{c:Lin-nar}.
\smallskip

If the completion of the uniform space $X$ is $\IR$-complete, then the equivalence  $(1)\Leftrightarrow(9)$ of Theorem~\ref{t:looong} yields the equivalence $(\U{\wedge}C)\Leftrightarrow(\U{\wedge}s)$. Combined with the implications $(\U{\wedge}s)\Ra(\Lin)\Ra(\Lc)\Ra(\U{\wedge}C)$, this equivalence yields the equivalences
$(\U{\wedge}s)\Leftrightarrow (\Lin)\Leftrightarrow(\Lc)\Leftrightarrow(\U{\wedge}C)$.
\smallskip

If the uniform space $X$ is $\IR$-universal, then the equivalence  $(1)\Leftrightarrow(9)\Leftrightarrow(15)$ of Theorem~\ref{t:looong} yields the equivalences $(\U{\wedge}\sigma)\Leftrightarrow(\U{\wedge}s)\Leftrightarrow(\U{\wedge}C)$. Combined with the implications $(\U{\wedge}s)\Ra(\Lin)\Ra(\Lc)\Ra(\U{\wedge}C)$, these equivalences yield the equivalences
$(\U{\wedge}\sigma)\Leftrightarrow(\U{\wedge}s)\Leftrightarrow (\Lin)\Leftrightarrow(\Lc)\Leftrightarrow(\U{\wedge}C)$.
\end{proof}

The following example (that can be derived from Example~\ref{ex:Lc-conv} and Corollary~\ref{c:Lin-nar}) shows that the implication $(\Lin)\Ra(\Lc)$ in Theorem~\ref{t:Lin+LcU} cannot be reversed.

\begin{example} For any non-separable convex set $X$ in a Banach space the free locally convex space $\Lc_u(X)$ has an $\w^\w$-base but the free topological vector space $\Lin_u(X)$ does not.
\end{example}

\begin{problem} Is a uniform space $X$ separable if its free topological vector space $\Lin_u(X)$ has an $\w^\w$-base?
\end{problem}

The answer to this problem is affirmative if $\w_1<\mathfrak b$ or if the completion of the uniform space $X$ is $\IR$-complete.
\smallskip

Applying Theorem~\ref{t:Lin+LcU}  to Tychonoff spaces endowed with their universal uniformity, we get the following characterization.

\begin{theorem}\label{t:Lin+LcT} For a Tychonoff space $X$ the following conditions are equivalent:
\begin{itemize}
\item[$(\Lin)$] the free topological vector space $\Lin(X)$ has an $\w^\w$-base;
\item[$(\Lc)$] the free locally convex space $\Lc(X)$ has an $\w^\w$-base;
\item[$(\U C)$] the universal uniformity $\U_X$ of $X$ has an $\w^\w$-base and the poset $C(X)$ is $\w^\w$-dominated;
\item[$(\U s)$] the universal uniformity $\U_X$ of $X$ has an $\w^\w$-base and the space $X$ is separable.
\item[$(\U\sigma)$] the universal uniformity $\U_X$ of $X$ has an $\w^\w$-base and the space $X$ is $\sigma$-compact.
%\item[$\sigma$] the space $X$ is $\sigma$-compact.
\end{itemize}
\end{theorem}

For an Ascoli space $X$ the equivalent conditions of Theorem~\ref{t:Lin+LcT} are equivalent to the existence of an $\w^\w$-base in the double function space $C_kC_k(X):=C_k(C_k(X))$.
We recall that $C_k(X)$ is the space of all real-valued continuous functions on $X$, endowed with the compact-open topology. A topological space $X$ is called {\em Ascoli} if for every compact subset $K\subset C_k(X)$ the evaluation function $e:K\times X\to \IR$, $e:(f,x)\mapsto f(x)$, is continuous. By \cite[5.4]{BG}, a topological space $X$ is Ascoli if and only if  the canonical map $\delta:X\to C_kC_k(X)$ assigning to each point $x\in X$ the Dirac measure $\delta_x:C_k(X)\to \IR$, $\delta_x:f\mapsto f(x)$, is continuous.

By \cite[\S5]{BG}, the class of Ascoli spaces includes all $k$-spaces and all Tychonoff $k_\IR$-spaces. A topological space $X$ is a {\em $k$-space} if a subset $U\subset X$ is open in $X$ if and only if for every compact subset $K\subset X$ the intersection $F\cap K$ is relatively open in $K$. A topological space $X$ is a {\em $k_\IR$-space} \cite{Noble} if a function $f:X\to\IR$ if continuous if and only if for any compact subset $K\subset X$ the restriction $f|K$ is continuous.

\begin{theorem}\label{t:Ascoli} For an Ascoli Tychonoff space $X$ the following conditions are equivalent:
\begin{enumerate}
\item[$(\Lc)$] the free locally convex space $\Lc(X)$ of $X$ has an $\w^\w$-base;
\item[$(\Lin)$] the free topological vector space $\Lin(X)$ of $X$ has an $\w^\w$-base;
\item[$(C_k^2)$] the double function space $C_kC_k(X)$ has a local $\w^\w$-base.
\end{enumerate}
\end{theorem}

\begin{proof} The equivalence of the first two conditions follows from Theorem~\ref{t:Lin+LcT}.

Now assume that the Tychonoff space $X$ is Ascoli. In this case every compact subset of $C_k(X)$ is equicontinuous, which implies that the space $C_kC_k(X)$ carries the topology of uniform convergence on pointwise bounded equicontinuous sets and hence $\Lc(X)\subset C_u^*(X)\subset C_kC_k(X)$.

If the space $C_kC_k(X)$ has an $\w^\w$-base, then so does its subspace $\Lc(X)$.

Now assume conversely that the free locally space $\Lc_u(X)$ has an $\w^\w$-base. By Theorem \ref{t:u} we have the equivalence $\E(C_u(X))\cong \w^{\w}$. The Ascoli property of $X$ implies that $\K(C_k(X))\cong \E(C_u(X))\cong\w^\w$. We obtained that
the function space $C_k(X)$ has a compact resolution swallowing compact sets and hence,
according to a theorem of Ferrando and K{\c{a}}kol \cite{feka},
the space $C_kC_k(X)$ has an $\w^\w$-base.
\end{proof}

%\begin{remark} By \cite[10.2.4]{Ban}, there exists a sequential countable $\mathfrak P_0$-space $X$ with a unique non-isolated point which fails to have an $\w^\w$-base. For this space $X$ the spaces $\Lc(X)$ and $\Lin(X)$ do not have $\w^\w$-bases.
%\end{remark}

\begin{remark} Theorem~\ref{t:Ascoli} answers Question 4.18 in \cite{GabKakLei_2}.
\end{remark}

\begin{remark} Free topological vector spaces $\Lin(X)$ over Tychonoff spaces are introduced and studied in a recent preprint \cite{GabMor}, but with completely different approach and results.
Note also that several results about the existence of $\w^\w$-bases in free locally convex spaces  were proved independently (and using different arguments than ours) in \cite{GK_L(X)}.
In particular, the authors of  \cite{GK_L(X)} proved that a metrizable space $X$ is $\sigma$-compact if and only if its free locally convex space $\Lc(X)$ has an $\w^\w$-base; also they proved that for a countable Ascoli space $X$ with an $\w^\w$-base, the free locally convex space $\Lc(X)$ has an $\w^\w$-base.
\end{remark}

\section{Strong topological groups over uniform spaces}\label{s:strong}

In this section we detect $\w^\w$-bases in strong topological groups over uniform spaces.

Given a topological group $G$ by $\Tau_e(G)$ we denote the family of neighborhoods of the unit $e$ of $G$, endowed with the partial order of reverse inclusion ($U\le V$ iff $V\subset U$). Each topological group $G$ will be endowed with the {\em two-sided uniformity} generated by the base consisting of the entourages $\{(x,y)\in G\times G:y\in xU\cap Ux\}$ where $U\in\Tau_e(G)$.
So, talking about uniformly continuous maps to a topological group $G$ we shall always have in mind the two-sided uniformity on $G$.

A topology $\tau$ on a group $G$ is called a {\em group topology} if it turns $G$ into a topological group; this happens if and only if the map $G\times G\to G$, $(x,y)\mapsto xy^{-1}$, is continuous.

\begin{definition} By a {\em strong topological group} over a uniform space $X$ we understand a pair $(G,\delta)$ consisting of a topological group $G$ and a uniformly continuous map $\delta:X\to G$ such that the image $\delta(X)$ algebraically generates $G$ and $G$ carries the strongest topology among all group topologies $\tau$ on $G$ making the map $\delta:X\to(G,\tau)$ uniformly continuous.
\end{definition}

Typical examples of strong topological groups over uniform spaces are the free (Abelian/Boolean) topological groups $(\AG_u(X),\delta_X)$, $(\BG_u(X),\delta_X)$, $(\FG_u(X),\delta_X)$.

Our aim is to describe the topology of a strong topological group $(G,\delta)$ over a uniform space $X$ in terms of the uniformity of the  space $X$.

For a group $G$ and sets $A,B\subset G$ let $A^{-1}=\{a^{-1}:a\in A\}$ be the inversion of the set $A$ in $G$, and let $A\cdot B=\{ab:a\in A,\;b\in B\}$ be the product of the sets $A,B$ in $G$. For a sequence $(A_n)_{n\in\w}$ of sets in $G$ the products $A_0\cdots A_n$ are defined inductively as $A_0\cdots A_n=(A_0\cdots A_{n-1})\cdot A_n$. For a finite ordinal $n$ by $S_n$ we denote the permutation group of the set $n=\{0,\dots,n-1\}$.
The {\em permutation product} of sets $A_0,\dots,A_{n-1}$ in a group $G$ is defined as the union
$$\textstyle{\bigotimes_{i\in n}A_i=\bigcup_{\pi\in S_n}A_{\pi(0)}\cdots A_{\pi(n)}.}$$
For a sequence $(A_n)_{n\in\w}$ of sets of a group $G$ let
$$\textstyle{\bigotimes_{n\in\w}A_n=\bigcup_{n=1}^\infty{\bigotimes_{i\in n}A_i}}$$
be the {\em permutation product} of the sequence $(A_n)_{n\in\w}$.

The following description of the topology of a strong topological group over a uniform space is a small improvement of the known description of the topology of a free topological group, due to V.~Pestov \cite{P} (see also \cite{RD}, \cite{Sipa1}, \cite{LPT}).

\begin{theorem}\label{t:strong} Let $(G,\delta)$ be a strong topological group over a uniform space $X$, $D$ be a dense subset of $X$ and $H$ be the subgroup of $G$ generated by the set $\delta(D)$. For an entourage $U\in\U_X$ let $$U^\pm=\big\{\delta(x)^\e\delta(y)^{-\e}:(x,y)\in U,\;\e\in\{-1,1\}\big\}\subset G.$$
Then the family
$$\textstyle{\mathcal B=\Big\{\bigotimes_{n\in\w}\bigcup_{g\in H}g U_{g,n}^\pm g^{-1}:(U_{g,n})_{(g,n)\in H\times\w}\in \U_X^{H\times \w}\Big\}}$$is a neighborhood base at the unit in the topological group $G$. %Consequently, $\U(X)^G)^\w\succcurlyeq \Tau_e(G)$.
\end{theorem}

\begin{proof}
 Let $\Tau$ be the topology of the topological group $G$ and $\tau$ be the topology on $G$ consisting of all sets $U\subset G$ such that for every $x\in U$ there exists a set $B\in\mathcal B$ with $xB\subset U$.

 First we show that $\tau$ is a group topology on $G$. For this it suffices to check that $\mathcal B$ satisfies the following Pontryagin' axioms (see \cite[1.3.12]{AT}):
\begin{enumerate}
\item for every $U\in\mathcal B$ there is $V\in\mathcal B$ such that $VV\subset U$ and $V^{-1}\subset U$;
\item for every $U\in\mathcal B$ and $u\in U$ there is $V\in\mathcal B$ such that $uV\subset U$;
\item for every $U\in\mathcal B$ and $z\in G$ there is $V\in\mathcal B$ such that $zVz^{-1}\subset U$;
\item for every $U,V\in\mathcal B$ there exists $W\in\mathcal B$ such that $W\subset U\cap V$.
%\item $\bigcap\mathcal B=\{e\}$.
\end{enumerate}
\smallskip

1. Given a set $U\in\mathcal B$, find $(U_{g,n})_{(g,n)\in H\times\w}\in\U_X^{H\times\w}$ such that $U=\bigotimes_{n\in\w}\bigcup_{g\in H}gU^\pm_{g,n}g^{-1}$. Consider the sequence $(V_{g,n})_{(g,n)\in H\times\w}\in \U_X^{H\times\w}$ defined by $V_{g,n}=\bigcap_{k\le 2n+1}U_{g,k}$ and observe that the basic neighborhood $V=\bigotimes_{n\in\w}\bigcup_{g\in H}gV^\pm_{g,n}g^{-1}\in\mathcal B$ has the required properties: $V=V^{-1}$ and $VV\subset U$.
\smallskip

2. Given a set $U\in\mathcal B$, find a sequence $(U_{g,n})_{(g,n)\in H\times \w}\in\U_X^{H\times\w}$  such that $U=\bigotimes_{n\in\w}\bigcup_{g\in H}gU^\pm_{g,n}g^{-1}$. For any $u\in U$ we can find $m\in\w$ such that $u\in\bigotimes_{n<m}\bigcup_{g\in H}gU^\pm_{g,n}g^{-1}$.
Consider the sequence $(V_{g,k})_{(g,k)\in H\times\w}$ defined by $V_{g,k}=U_{g,m+k}$ and observe that the basic neighborhood $V=\bigotimes_{k\in\w}\bigcup_{g\in H}gV^\pm_{g,k}g^{-1}$ has the required property: $uV\subset U$.
\smallskip

3. Given any set $U\in\mathcal B$ and point $y\in G$, we should find a set $W\in\mathcal B$ such that $yWy^{-1}\subset U$. Since the group $G$ is algebraically generated by the set $\delta(X)$, it suffices to consider the case $y\in\delta(X)\cup\delta(X)^{-1}$. Find a point $y'\in X$ such that $y=\delta(y')^\e$ for some $\e\in\{-1,1\}$.

By the first item, there exists a set $V\in\mathcal B$ such that $V=V^{-1}$ and $VVV\subset U$. Find a sequence $(V_{g,n})_{(g,n)\in H\times\w}\in\U_X^{H\times\w}$ such that $V=\bigotimes_{n\in\w}\bigcup_{g\in H}gV^\pm_{g,n}g^{-1}$.

Let $e$ be the unit of the group $H$. By the density of $D$ in $X$ the entourage $V_{e,0}\cap V_{e,1}$ contains a pair $(y',z')$ for some $z'\in D$. Then for the point $z=\delta(z')\in H$ we get $y=yz^{-1}z\in V_{e,0}^\pm z\cap V_{e,1}^\pm z$. We claim that the basic entourage $W=\bigotimes_{n\in\w}\bigcup_{g\in H}gV_{zg,n+2}g^{-1}\in\mathcal B$ has the required property: $yWy^{-1}\subset U$. First observe that
$$\textstyle{zWz^{-1}=\bigotimes_{n\in\w}\bigcup_{g\in H}zgV^\pm_{zg,n+2}g^{-1}z^{-1}=\bigotimes_{n\in\w}\bigcup_{h\in H}hV^\pm_{h,n+2}h^{-1}}$$and hence
$$\textstyle{yWy^{-1}\subset V_{e,0}^\pm zWz^{-1}V_{e,1}^\pm\subset \bigotimes_{n\in\w}\bigcup_{h\in H}hV^\pm_{h,n}h^{-1}=U.}$$

4. Given two sets $U,V\in\mathcal B$, find sequences $(U_{g,n})_{(g,n)\in H\times\w},(V_{g,n})_{(g,n)\in H\times\w}\in\U_X^{H\times \w}$ such that $U=\bigotimes_{n\in\w}\bigcup_{g\in H}gU^\pm_{g,n}g^{-1}$ and $V=\bigotimes_{n\in\w}\bigcup_{g\in G}gV^\pm_{g,n}g^{-1}$. Then for the entourages $W_{g,n}=U_{g,n}\cap V_{g,n}$, $(g,n)\in H\times\w$, and the set $W=\bigotimes_{n\in\w}\bigcup_{g\in H}gW^\pm_{g,n}g^{-1}\in\mathcal B$, we get the required inclusion $W\subset U\cap V$.
\smallskip

Therefore, the family $\mathcal B$ satisfies the Pontryagin's axioms and hence $\mathcal B$ is a neighborhood base of the group topology $\tau$ at the unit of the group $G$. Next, we show that the map $\delta:X\to (G,\tau)$ is uniformly continuous.
Given any neighborhood $B\in\mathcal B\subset\tau_e(G)$ of $e$ we should check that the entourage $\{(x,y)\in X\times X:\delta(y)\in B\cdot \delta(x)\cap \delta(x)\cdot B\}$ belongs to the uniformity $\U_X$ of $X$. Find a sequence $(B_{g,n})_{(g,n)\in H\times\w}\in\U_X^{H\times\w}$ such that  $B=\bigotimes_{n\in\w}\bigcup_{g\in G}gB^\pm_{g,n}g^{-1}$. Consider the entourage $V=B_{e,0}$ and observe that for every $(x,y)\in V$ we get $\{\delta(x)\delta(y)^{-1},\delta(x)^{-1}\delta(y)\}\subset V^\pm\subset B$, which implies $\delta(y)\in B\cdot\delta(x)\cap \delta(x)\cdot B$ and witnesses that the map $\delta:X\to (G,\tau)$ is uniformly continuous.

Since $G$ is a strong topological group over $X$, its topology $\Tau$ contains the group topology $\tau$ (so $\tau\subset\Tau)$. It remains to check that  $\Tau\subset \tau$.
Given any neighborhood $V\in\Tau_e(G)$, find a decreasing sequence of neighborhoods $\{V_n\}_{n\in\w}\subset\Tau_e(G)$ of the unit in $G$ such that $V_n^{-1}=V_n$ and $V_{n+1}^3\subset V_n$ for all $n\in\w$, and $V_0^3\subset U$. Such a choice of the sequence $(V_n)_{n\in\w}$ guarantees that $\bigotimes_{n\in\w}V_n\subset U$. For every $(g,n)\in H\times\w$ the set $g^{-1}V_ng$ is a neighborhood of the unit in the topological group $G$. The uniform continuity of the map $\delta:X\to G$ guarantees that the entourage $U_{g,n}=\{(x,y)\in X\times X:\{\delta(x)\delta(y)^{-1},\delta(x)^{-1}\delta(y)\}\subset g^{-1}V_ng\}$ belongs to the uniformity $\U_X$ of $X$. The definition of $U_{g,n}$ guarantees that $U_{g,n}^\pm\subset g^{-1}V_ng$ and hence the basic neighborhood
$$\textstyle{\bigotimes_{n\in\w}\bigcup_{g\in H}gU_{g,n}^\pm g^{-1}\subset \bigotimes_{n\in\w}V_n\subset U}$$is contained in $U$, which implies that $\Tau\subset\tau$ and finally $\Tau=\tau$.
\end{proof}

\begin{corollary}\label{c:strong-dens}  If $(G,\delta)$ is a strong topological group over a uniform space $X$ of density $\kappa$, then $\U_X^\kappa\succcurlyeq \Tau_e(G)$.
\end{corollary}

\begin{corollary}\label{c:strong-sep}  Let $(G,\delta)$ be a strong topological group over an $\w^\w$-based uniform space $X$. If the space $X$ is separable, then the topological group $G$ has an $\w^\w$-base.
\end{corollary}

Theorem~\ref{t:strong} has a modification for strong balanced topological groups over uniform spaces. A topological group $G$ is {\em balanced} if for every neighborhood $U\in\Tau_e(G)$ there exists a neighborhood $V\in\Tau_e(G)$ such that $gVg^{-1}\subset U$ for all $g\in G$. A group topology $\tau$ on a group $G$ will be called {\em balanced} if the topological group $(G,\tau)$ is balanced.

\begin{definition} A {\em strong balanced topological group} over a uniform space $X$ is a pair $(G,\delta)$ consisting of a balanced topological group $G$ and a uniformly continuous map $\delta:X\to G$ such that the image $\delta(X)$ algebraically generates $G$ and $G$ carries the strongest topology among all balanced group topologies $\tau$ on $G$ making the map $\delta:X\to(G,\tau)$ uniformly continuous.
\end{definition}

The following description of the topology of a strong balanced topological groups can be proved by analogy with Theorem~\ref{t:strong}, see also \cite{LPT}.

\begin{theorem}\label{t:strong-balance} For a strong balanced topological group $(G,\delta)$ over a uniform space $X$ the family
$$\mathcal B=\big\{\textstyle{\bigotimes}_{n\in\w}\bigcup_{g\in G}gU_n^\pm g^{-1}:(U_n)_{n\in\w}\in \U_X^\w\}$$is a neighborhood base at the unit of the topological group $G$. Consequently, $\U_X^\w\succcurlyeq \Tau_e(G)$.
\end{theorem}

In this theorem for an entourage $U\in\U_X$ by $U^\pm$ we denote the set
$$\big\{\delta(x)^\e\delta(y)^{-\e}:(x,y)\in U,\;\e\in\{-1,1\}\big\}\subset G.$$

Theorem~\ref{t:strong-balance} implies the following criterion.

\begin{corollary} For a strong balanced topological group $(G,\delta)$ over a uniform space $X$ the topological group $G$ has an $\w^\w$-base if the uniformity of $X$ has an $\w^\w$-base.
\end{corollary}

Finally we consider strong (balanced) topological groups over uniform $P_u$-spaces.  Observe that a uniform space $X$ is a uniform $P_u$-space if and only if $\add(\U_X)>\w$.

A topological group $G$ is called a {\em topological $P$-group} if its underlying topological space is a $P$-space.
It is easy to see that a topological group $G$ is a $P$-group if and only if $G$ endowed with the two-sided uniformity is a uniform $P_u$-space.

\begin{corollary}\label{c:strong-P} Let $(G,\delta)$ be a strong topological group over a uniform $P_u$-space $X$. If $\cov^\sharp(X)\le \add(X)$, then the topological group $G$ is balanced and the family $$\mathcal B=\big\{\textstyle{\bigotimes_{n\in\w}\bigcup_{g\in G}gU^\pm g^{-1}:U\in\U_X}\big\}$$ is a neighborhood base at the unit $e$ of $G$. Consequently $G$ is a $P$-group and $\U_X\succcurlyeq \Tau_e(G)$.
\end{corollary}

\begin{proof} The statement is trivial if the uniform $P_u$-space $X$ is discrete (and then $G$ is discrete). So, we assume that $X$ is not discrete and hence $\w<\add(X)<\infty$.
Given a neighborhood $U\in\tau_e(G)$ of the unit $e$ of $G$, we shall find a basic set $B\in\mathcal B$ with $B\subset U$. Using the continuity of multiplication and inversion at $e$, choose a sequence of neighborhoods $\{U_n\}_{n\in\w}\subset\Tau_e(G)$ such that $U_0^3\subset U_{-1}:= U\cap U^{-1}$ and $U_{n+1}^3\subset U_n=U_{n}^{-1}$ for all $n\in\w$. It follows that $H=\bigcap_{n\in\w}U_n$ is a subgroup of $G$. By $\U_G$ we denote the two-sided uniformity of the topological group $G$.

 By \cite[513C]{F-MT}, the cardinal $\kappa=\add(X)>\w$ is regular.
   The uniform continuity of the canonical map $\delta:X\to G$ implies that $\cov^\sharp(\delta(X);\U_G)\le \cov^\sharp(X)\le\add(X)=\kappa$. Since the group $G$ is algebraically generated by the set $\delta(X)$, we can use the regularity of the cardinal $\kappa$ and repeating the proof of Theorem 5.1.19 of \cite{Tk} can show that $\cov^\sharp(G;\U_G)\le \cov^\sharp(\delta(X);\U_G)\le\kappa$. Then for every $n\in\w$ there exists a set $C_n\subset G$ of cardinality $|C_n|<\kappa$ such that $G=C_nU_{n}$. For every $x\in C_n$ choose a neighborhood $W_{n,x}\in\Tau_e(G)$ such that $W_{n,x}x\subset xU_{n}$. Then for every $y\in xU_{n}$ we get $$W_{n,x}y\subset W_{n,x}xU_{n}\subset xU_{n}U_{n}\subset yU_{n}^{-1}U_nU_n\subset yU_{n-1}.$$
 It follows that the family $\V=\{W_{n,x}:n\in\w,\;x\in C_n\}$ has cardinality $|\V|\le\sum_{n\in\w}|C_n|<\kappa$ and for any $n\in\w$ and $x\in G$ there exists a neighborhood $V_{n,x}\in\V$ such that $V_{n,x}x^{-1}\subset x^{-1}U_n$. By the uniform continuity of the canonical map $\delta:X\to G$, for every $V\in\V$ the entourage $$\tilde V=\{(x,y)\in X\times X:\delta(x)\in \delta(y)V\cap V\delta(y)\}=\{(x,y)\in X\times X:\{\delta(x)\delta(y)^{-1},\delta(x)^{-1}\delta(y)\}\subset V\}$$ belongs to the uniformity $\U_X$ of $X$. Since $|\{\tilde V\}_{V\in\V}|\le|\V|<\kappa=\add(X)$, the intersection $W=\bigcap_{V\in\V}\tilde V$ belongs to the uniformity $\U_X$. Observe that for every $V\in\V$ we get $W^\pm\subset \tilde V^\pm\subset V$. Consequently, for every $n\in\w$ and $x\in G$ we get $xW^\pm x^{-1}\subset xV_{n,x}x^{-1}\subset U_n$ and hence $\bigcup_{x\in G}xW^\pm x^{-1}\subset \bigcap_{n\in\w}U_n\subset H$. Since $H$ is a subgroup, we conclude that the set $B=\bigotimes_{n\in\w}\bigcup_{x\in G}xW^\pm x^{-1}\in\mathcal B$ is contained in $H\subset U$. This implies that the (balanced) group topology $\tau_b$ on $G$ generated by the base $\mathcal B$ is stronger than the topology $\tau$ of the group $G$. Taking into account that $\tau$ of $G$ is the strongest group topology on $G$ making the canonical map $\delta:X\to G$ uniformly continuous, we conclude that $\tau=\tau_b$ and hence the topology $\tau$ is balanced and $\mathcal B$ is a neighborhood base of the topology $\tau$ at $e$.

To see that $G$ is a $P$-groups, choose a countable subfamily $\{U_n\}_{n\in\w}\subset \mathcal B$ and for every $n\in\w$ find an entourage $V_n\in\U_X$ such that $U_n=\bigotimes_{m\in\w}\bigcup_{g\in G}gV_n^\pm g^{-1}$. Since $X$ is a uniform $P_u$-space, the intersection $V=\bigcap_{n\in\w}V_n$ belongs to the uniformity $\U_X$. Then the neighborhood $U=\bigotimes_{n\in\w}\bigcup_{g\in G}g V g^{-1}\in\mathcal B$ of the unit $e$ is contained in the intersection $\bigcap_{n\in\w}U_n$, witnessing that $G$ is a topological $P$-groups.
 \end{proof}

\begin{theorem}\label{t:sG-ca} Let $(G,\delta)$ be a strong topological group over a uniform space $X$. If $X$ is $\w^\w$-based and $\cov^\sharp(X)\le\add(X)$, then the topological group $G$ has an $\w^\w$-base.
\end{theorem}

\begin{proof} Assume that the uniform space $X$ is $\w^\w$-based and $\cov^\sharp(X)\le\add(X)$. If $X$ is a uniform $P_u$-space, then by Corollary~\ref{c:strong-P}, $\w^\w\succcurlyeq \U_X\succcurlyeq\Tau_e(G)$, which means that the topological group $G$ has an $\w^\w$-base.

It remains to consider the case when $X$ is not a $P_u$-space and hence $\add(X)=\w$. In this case the inequality $\cov^\sharp(X)\le\add(X)=\w$ means that the uniform space $X$ is totally bounded and hence has compact completion $\bar X$. By \cite[6.8.4]{Ban}, the uniform space $\bar X$ is $\w^\w$-based and by Theorem~\ref{t:Sigma}, the compact space $\bar X$ is metrizable and hence separable. Then it subspace $X$ is separable too, and by Corollary~\ref{c:strong-sep} the topological group $G$ has an $\w^\w$-base.
\end{proof}

\section{Free Abelian (Boolean) topological groups}

For a uniform space $X$ its {\em free Abelian topological group} is a pair $(\AG_u(X),\delta_X)$ consisting of an Abelian topological group $\AG_u(X)$ and a uniformly continuous map $\delta_X:X\to \AG_u(X)$ such that for every uniformly continuous map $f:X\to Y$ into an Abelian topological group $Y$ there exists a continuous group homomorphism $\bar f:\AG_u(X)\to Y$ such that $\bar f\circ\delta_X=f$.

Replacing the adjective ``Abelian'' by ``Boolean'' in this definition, we get the definition of a {\em free Boolean topological group} $(\BG_u(X),\delta_X)$ over a uniform space $X$.

Observe also that for any uniform space $X$ the free Abelian and Boolean topological groups $(\AG_u(X),\delta_X)$ and $(\BG_u(X),\delta_X)$ are strong balanced topological groups over $X$. This observation allows us to apply Theorem~\ref{t:strong-balance} to describing the topological structure of the groups $\AG_u(X)$ and $\BG_u(X)$.

\begin{theorem}\label{t:AGu} Let $X$ be a uniform space. For the subspace $$X-X=\{x-y:x,y\in \delta_X(X)\}\subset\AG_u(X)$$ the following reductions hold:
$$\U(X)^\w\succcurlyeq\Tau_0(\AG_u(X))\succcurlyeq \Tau_0(X-X)\succcurlyeq \U_X.$$
\end{theorem}

\begin{proof} The reduction $\U_X^\w\succcurlyeq\Tau_0(\AG_u(X))$ follows from Theorem~\ref{t:strong-balance} and the observation that $(\AG_u(X),\delta_X)$ is a strong balanced topological group over $X$. The reduction $\Tau_0(\AG_u(X))\succcurlyeq \Tau_0(X-X)$ trivially follows from the inclusion $X-X\subset \AG_u(X)$.

To prove that $\Tau_0(X-X)\succcurlyeq \U_X$, consider the monotone map $E_*:\Tau_0(X-X)\to \U_X$ assigning to each neighborhood $V\in \Tau_0(X-X)$ the entourage $E_V=\{(x,y)\in X\times X:x-y\in V\}$. Here we identify the set $X$ with its image $\delta_X(X)$ in $\AG_u(X)$. The uniform continuity of the map $\delta_X:X\to\AG_u(X)$ guarantees that the entourage $E_V$ belongs to the uniformity $\U_X$ of $X$. It remains to prove that the map $E_*$ is cofinal. Given any entourage $U\in\U_X$, find a uniform pseudometric $d$ on $X$ such that $\{(x,y)\in X\times X:d(x,y)<1\}\subset U$. By  Graev's Theorem 7.2.2 \cite{AT}, the pseudometric $d$ can be extended to an invariant pseudometric $\hat d$ on the free group $\AG_u(X)$. Then for the unit ball $V=\{g\in\AG_u(X):\hat d(g,e)<1\}$ we get
$$
%\begin{aligned}
E_V=\{(x,y)\in X\times X:x-y\in V\}=\{(x,y)\in X\times X:\hat d(x,y)<1\}
=\{(x,y)\in X\times X:d(x,y)<1\}\subset U
%\end{aligned}
$$and we are done.
\end{proof}

By analogy we can prove a similar reduction result for the free Boolean topological groups.

\begin{theorem}\label{t:BGu} For every uniform space $X$ and the subspace $$X+X=\{x+y:x,y\in \delta_X(X)\}\subset\BG_u(X)$$ we have the reductions:
$$\U_X^\w\succcurlyeq\Tau_0(\BG_u(X))\succcurlyeq \Tau_0(X+X)\succcurlyeq \U_X.$$
\end{theorem}

Unifying Theorems~\ref{t:AGu} and \ref{t:BGu} and taking into account that $(\w^\w)^\w\cong\w^\w$, we get the following characterization, whose ``abelian'' part was found also in \cite{GartMor}, \cite{LPT}, \cite{LT}.

\begin{theorem}\label{t:ABGu} For a uniform space $X$ the following conditions are equivalent:
\begin{enumerate}
\item The free Abelian topological group $\AG_u(X)$ has an $\w^\w$-base.
\item The free Boolean topological group $\BG_u(X)$ has an $\w^\w$-base.
\item The subspace $X-X=\{x-y:x,y\in\delta_X(X)\}\subset \AG_u(X)$ has a neighborhood $\w^\w$-base at zero.
\item The subspace $X+X=\{x+y:x,y\in\delta_X(X)\}\subset \BG_u(X)$ has a neighborhood $\w^\w$-base at zero.
\item The uniformity $\U_X$ has an $\w^\w$-base.
\end{enumerate}
\end{theorem}

 Taking into account that for a Tychonoff space $X$ the free Abelian and Boolean topological groups $\AG(X)$ and  $\BG(X)$ coincide with the free Abelian and Boolean topological groups $\AG_u(X)$ and $\BG_u(X)$ of the space $X$ endowed with the universal uniformity $\U_X$, we get the following characterization of Tychonoff spaces whose free Abelian and Boolean topological groups have an $\w^\w$-base.

\begin{corollary}\label{c:ABG} For a Tychonoff space $X$ the following conditions are equivalent:
\begin{enumerate}
\item The free Abelian topological group $\AG(X)$ of $X$ has an  $\w^\w$-base.
\item The free Boolean topological group $\BG(X)$ of $X$ has an $\w^\w$-base.
\item The subspace $X-X=\{x-y:x,y\in\delta_X(X)\}\subset \AG(X)$ has a neighborhood $\w^\w$-base at zero.
\item The subspace $X+X=\{x+y:x,y\in\delta_X(X)\}\subset \BG(X)$ has a neighborhood $\w^\w$-base at zero.
\item The universal uniformity $\U_X$ of $X$ has an $\w^\w$-base.
\end{enumerate}
\end{corollary}

The characterization Theorem~\ref{t:ABGu} has an interesting implication related to uniformly quotient maps.

\begin{definition} A surjective map $f:X\to Y$ between uniform spaces is {\em uniformly quotient} if an entourage $U\subset Y\times Y$ belongs to the uniformity $\U_Y$ if and only if its preimage $(f\times f)^{-1}(U)=\{(x,y)\in X\times X:(f(x),f(y))\in U\}$ belongs to the uniformity $\U_X$.
\end{definition}

This definition implies that each uniformly quotient map between uniform spaces is uniformly continuous.

\begin{lemma}\label{l:quot} For a uniformly quotient map $f:X\to Y$ between uniform spaces the induced group homomorphism $\bar f:\AG_u(X)\to\AG_u(Y)$ is open.
\end{lemma}

\begin{proof} Let $H=\bar f^{-1}(0)$ be the kernel of the homomorphism $\bar f$, $\AG_u(X)/H$ be the quotient topological Abelian group and $q:\AG_u(X)\to \AG_u(X)/H$ be the quotient homomorphism. It follows that $\bar f=\tilde f\circ q$ for some continuous bijective homomorphism $\tilde f:\AG_u(X)/H\to \AG_u(Y)$. The equality $\tilde f\circ q\circ\delta_X=\bar f\circ \delta_X=\delta_Y\circ f$ implies the equality $q\circ \delta_X=\tilde f^{-1}\circ\delta_Y\circ f$. Taking into account that the map $f$ is uniformly quotient and the map $q\circ \delta_X=\tilde f^{-1}\circ\delta_Y\circ f$ is uniformly continuous, we conclude that the map $\tilde f^{-1}\circ \delta_Y$ is uniformly continuous. Then the definition of the free topological abelian group $\AG_u(Y)$ implies that the homomorphism $\tilde f^{-1}:\AG_u(Y)\to \AG_u(X)/H$ is continuous. So, $\tilde f$ is a topological isomorphism. Since the quotient map $q:\AG_u(X)\to \AG_u(X)/H$ is open, so is the map $f=\tilde f\circ q:\AG_u(X)\to\AG_u(Y)$.
\end{proof}

\begin{corollary}\label{c:uniquot} Let $f:X\to Y$ be a uniformly quotient map between uniform spaces. If the uniform space $X$ is $\w^\w$-based, then so is the uniform space $Y$.
\end{corollary}

\begin{proof} By Lemma~\ref{l:quot}, the induced homomorphism $\bar f:\AG_u(X)\to\AG_u(Y)$ is open. If the uniformity $\U_X$ of $X$ has an $\w^\w$-base, then by Theorem~\ref{t:ABGu}, the free abelian topological group $\AG_u(X)$ has a local $\w^\w$-base $(U_\alpha)_{\alpha\in\w^\w}$ at zero. Since the homomorphism $\bar f$ is open, $(\bar f(U_\alpha))_{\alpha\in\w^\w}$ is a neighborhood $\w^\w$-base at zero of the group $\AG_u(Y)$. By Theorem~\ref{t:ABGu}, the uniformity $\U_Y$ has an $\w^\w$-base.
\end{proof}

A map $f:X\to Y$ between topological spaces is called {\em $\IR$-quotient} if for any function  $\varphi:Y\to \IR$ the continuity of $\varphi$ is equivalent to the continuity of the composition $\varphi\circ f:X\to\IR$.

\begin{proposition}\label{p:Rquot} Let $f:X\to Y$ be an $\IR$-quotient map of Tychonoff spaces. If the universal uniformity of $X$ has an $\w^\w$-base, then the universal uniformity of the space $Y$ has an $\w^\w$-base, too.
\end{proposition}

%{\color{red} This is Corolary 9.3 of \cite{BL}}.

\begin{proof} This proposition will follow from Corollary~\ref{c:uniquot} as soon as we check that the $\IR$-quotient map $f:X\to Y$ is uniformly quotient (with respect to the universal uniformities $\U_X$ and $\U_Y$). Given a continuous pseudometric $d_Y$ on $Y$ we need to show that the pseudometric $d_X=d_Y(f\times f)$ is continuous. By the triangle inequality, it suffices to show that for every point $x_0\in X$ the map $d_X(x_0,\cdot):X\to \IR$, $d_X(x_0,\cdot):x\mapsto d_X(x_0,x)=d_Y(f(x_0),f(x))$, is continuous. Since the map $f:X\to Y$ is $\IR$-quotient, the continuity of the map $d_Y(f(x_0),\cdot):Y\to\IR$ implies the continuity of the map $d_X(x_0,\cdot)=d_Y(f(x_0),\cdot)\circ f$.
\end{proof}

\section{Free topological groups}\label{s:FG}

For a uniform space $X$ its {\em free topological group} is a pair $(\FG_u(X),\delta_X)$ consisting of a topological group $\FG_u(X)$ and a uniformly continuous map $\delta_X:X\to \FG_u(X)$ such that for every uniformly continuous map $f:X\to Y$ to a topological group $Y$ there exists a continuous group homomorphism $\bar f:\FG_u(X)\to Y$ such that $\bar f\circ\delta_X=f$.  This definition implies that the pair $(\FG_u(X),\delta_X)$ is a strong topological group over the uniform space $X$, so we can apply the results of Section~\ref{s:strong} to studying the structure of the group $\FG_u(X)$.

For a Tychonoff space $X$ its free topological group $(\FG(X),\delta_X)$ coincides with the free  topological group $(\FG_u(X),\delta_X)$ of the space $X$ endowed with the universal uniformity $\U_X$.

For a uniform space $X$ its free topological group $(\FG_u(X),\delta_X)$ can be constructed as follows. Take the free group $(\FG(X),\delta_X)$ over the set $X$. It is well-known that in this case the map $\delta_X:X\to \FG(X)$ is injective, so $X$ can be identifying with the subset $\delta_X(X)$ (which algebraically generates the group $\FG(X)$). By Graev's Theorem 7.2.2 in \cite{AT}, every (pseudo)metric $d$ on $X$ admits a canonical extension to an invariant (pseudo)metric on $\FG_u(X)$. This implies that the family of Graev's extensions of uniform pseudometrics on the uniform space $X$ separate points of the free group $\FG(X)$ and hence determine a Hausdorff group topology $\tau$ on $\FG(X)$ such that the map $\delta_X:X\to\FG(X)$ is uniformly continuous. Now we can take the strongest group topology $\Tau$ on $\FG(X)$ for which the map $\delta_X:X\to \FG(X)$ uniformly continuous. The group $\FG(X)$ endowed with this strongest topology is denoted by $\FG_u(X)$. It follows that the pair $(\FG_u(X),\delta_X)$ is a free topological group over the uniform space $X$. This construction implies that the map $\delta_X:X\to\FG_u(X)$ is a topological embedding and the group $\FG_u(X)$ is algebraically free over $X$. For a point $z\in X$ by $z_\delta$ we shall denote its image $\delta_X(z)$ in $\FG_u(X)$.

We recall that a uniform space $X$ is {\em $\kappa$-universal} if every $\kappa$-continuous map $f:X\to M$ to a metric space $M$ of density $\kappa$ is uniformly continuous. A family $\mathcal D$ of subsets of a topological space $X$ is {\em discrete} if each point $x\in X$ has a neighborhood $O_x\subset X$ meeting at most one set of the family $\mathcal D$.

\begin{theorem}\label{t:FGu} For every uniform space $X$ of density $\kappa$ the following reductions hold:
\begin{enumerate}
\item $\Tau_e(XX^{-1})\succcurlyeq \U_X$ where $XX^{-1}=\{xy^{-1}:x,y\in\delta_X(X)\}\subset\FG_u(X)$.
\item $\U_X^{\kappa}\succcurlyeq \Tau_e(\FG_u(X))$.
\item If $X$ is a uniform $P_u$-space with $\cov^\sharp(X)\le\add(X)$, then $\U_X\succcurlyeq \Tau_e(\FG_u(X))$.
\item If the uniform space $X$ is $\w$-universal and $X$ contains a non-$P$-point $z\in X$, then  $\Tau_e(Y_z)\succcurlyeq B_\w(X)$, where $Y_z=\{xy z_\delta^{-1}x^{-1}:x,y\in \delta_X(X)\}\subset \FG_u(X)$.
\item If $\mathcal D$ is a discrete family of non-empty open sets in $X$, the uniformity of $X$ is $|\mathcal D|\cdot\add(\U_X)$-universal, and $z\in X$ is a non-isolated point with  $\add(\Tau_z(X))=\add(X)>\w$, then  $\Tau_e(Y_z)\succcurlyeq\add(X)^{\mathcal D}$, where $Y_z=\{xyz_\delta^{-1}x^{-1}:x,y\in \delta_X(X)\}\subset\FG_u(X)$.
\end{enumerate}
\end{theorem}

\begin{proof} 1. The first statement can be proved by analogy with Theorem~\ref{t:AGu} (using Graev's extensions of pseudometrics, see \cite[7.2.2]{AT}).
\smallskip

2,3. The second and third statements follow from Corollaries~\ref{c:strong-dens} and \ref{c:strong-P}, respectively.
\smallskip

4. To prove the fourth statement, assume that the uniformity of $X$ is $\w$-universal and $X$ contains a non-$P$-point $z$. Then there exists a decreasing sequence of entourages $\{U_n\}_{n\in\w}\subset \U_X$ such that the intersection $\bigcap_{n\in\w}U_n[z]$ is not a neighborhood of the point $z$ in $X$. We can additionally assume that $U_{n+1}^3\subset U_n=U_n^{-1}$ for all $n\in\w$. By Theorem~8.1.10 \cite{Eng}, there exists a pseudometric $d\le 1$ on $X$ such that
$$\{(x,y):d(x,y)<2^{-n}\}\subset U_n\subset \{(x,y):d(x,y)\le 2^n\}$$for every $n\in\IN$.

Consider the monotone map  $\varphi_*:\Tau_e(Y_z)\to \w^X$, assigning to each neighborhood $U\in\Tau_e(Y_z)$ of $e$ the function $$\varphi_U:X\to\w,\;\;\varphi_U:x\mapsto
\min\{n\in\IN: \exists y\in X\setminus U_n[z]\mbox{ such that }xyz^{-1}x^{-1}\in U\}.$$
Here we identify the set $X$ with its image $\delta_X(X)$ in $\FG_u(X)$.
Let us show that the number $\varphi_U(x)$ is well-defined. Choose a neighborhood $\tilde U\subset\FG_u(X)$ such that $\tilde U\cap Y_z=U$. The uniform continuity of the map $\delta_X:X\to\FG_u(X)$ ensures that the entourage
$$V_x=\{(u,v)\in X\times X:vu^{-1}\in x^{-1}\tilde Ux\}=\{(u,v)\in X\times X:xvu^{-1}x^{-1}\in U\}$$belongs to the uniformity $\U_X$.
The choice of the sequence $(U_n)_{n\in\w}$ guarantees that $V_x[z]\not\subset U_n[z]$ for some $n\in\IN$. Consequently, there exists a point $y\in V_x[z]\setminus U_n[z]$. For this point we get $xyz^{-1}x^{-1}\in U$, witnessing that $\varphi_U(x)\le n$.

Next, we show that the function $\varphi_U:X\to\w$ is locally $\w$-bounded. Choose a neighborhood $V\in\Tau_e(\FG_u(X))$ such that  $VVV^{-1}\subset\tilde U$.
The uniform continuity of the map $\delta_X:X\to\FG_u(X)$ ensures that the entourage $V'=\{(x,y)\in X\times X:y\in Vx\}$ belongs to the uniformity $\U_X$. By the definition of the number $n=\varphi_{V\cap Y_z}(x)$, there exists a point $y\in X\setminus U_n[z]$ such that $xyz^{-1}x^{-1}\in V$. Then for every $v\in V'[x]=Vx\cap X$ we get $vyz^{-1}v^{-1}\in Vxyz^{-1}x^{-1}V^{-1}\subset VVV^{-1}\subset \tilde U$, which implies $\varphi_U(v)\le n$ and hence $\varphi_U(V'[x])\subset [0,n]$.
This means that the map $\varphi_U$ is locally $\w$-bounded and hence $\varphi_U\in B_\w(X)$.

It remains to show that the monotone map $\varphi_*:\Tau_e(Y_z)\to B_\w(X)$, $\varphi_*:U\mapsto\varphi_U$, is cofinal. Given any locally $\w$-bounded function $\psi\in B_\w(X)$, we need to find a neighborhood $U\in\Tau_e(Y_z)$ such that $\varphi_U\ge \psi$. By Lemma~\ref{l:CwB}, there exists a $\w$-continuous function $\gamma:X\to[0,\infty)$ such that $\gamma(x)\ge \psi(x)+1$ for all $x\in X$. Let $\IR_+=(0,\infty)$ and $\alpha:X\to\IR_+$ be defined by $\alpha(x)=2^{\gamma(x)}$.
Consider the uniformly continuous function $\beta:X\to[0,1]\subset \IR$, $\beta:x\mapsto d(x,z)$.

Denote by $\IR$ and $\IR_+$ the additive group of real numbers and the multiplicative group of positive real numbers, respectively. On the product $\IR\times\IR_+$ consider the group operation $(x,a)*(y,b)=(x+ay,ab)$.
The obtained topological group will be denoted by $\IR\rtimes\IR_+$ and called the {\em semidirect product} of the topological groups $\IR$ and $\IR_+$. For any element $(x,a)\in\IR\times\IR_+$ its inverse $(x,a)^{-1}$ in the group $\IR\rtimes\IR_+$ is equal to $(-\frac{x}{a},\frac1{a})$.

The maps $\alpha,\beta$ determine a $\w$-continuous map $\theta:X\to \IR\rtimes\IR_+$, $\theta:x\mapsto(\beta(x),\alpha(x))$. Since the uniformity of $X$ is $\w$-universal, the $\w$-continuous map $\theta:X\to\IR\rtimes\IR_+$ is uniformly continuous and hence admits a unique extension to a continuous group homomorphism $\bar\theta:\FG_u(X)\to\IR\times\IR_+$. Consider the open neighborhood $$U_0=\{(x,y)\in\IR\rtimes\IR_+:|x|<1,\;|1-y|<\tfrac12\}$$ of the neutral element $(0,1)$ of the group $\IR\rtimes\IR_+$ and the preimage $U=\bar\theta^{-1}(U_0)\subset\FG_u(X)$ of $U_0$ under the homomorphism $\bar\theta$.
We claim that $\varphi_U\ge \psi$. To derive a contradiction, assume that $\varphi_U(x)<\psi(x)$ for some $x\in X$. Then for the number $n=\varphi_U(x)$ there exists a point $y\in X\setminus U_n[z]$ such that $xyz^{-1}x^{-1}\in U$. It follows that $(z,y)\notin U_n$ and hence $\beta(y)=d(y,z)\ge 2^{-n}$.

In the group $\IR\rtimes\IR_+$ consider the element $$
\begin{aligned}
\bar\theta(xyz^{-1}x^{-1})&=\theta(x)*\theta(y)*\theta(z)^{-1}*\theta(x)^{-1}=
(\beta(x),\alpha(x))*(\beta(y),\alpha(y))*(0,{\alpha(z)})^{-1}*(\beta(x),\alpha(x))^{-1}=\\
&=\big(\beta(x)+\alpha(x)\beta(y)-\tfrac{\alpha(y)\beta(x)}{\alpha(z)},\tfrac{\alpha(y)}{\alpha(z)}\big)=\big(\beta(x)(1-\tfrac{\alpha(y)}{\alpha(z)})+\alpha(x)\beta(y),\tfrac{\alpha(y)}{\alpha(z)}\big).
\end{aligned}$$

The inclusion $xyz^{-1}x^{-1}\in U$ implies that $$
\big(\beta(x)(1-\tfrac{\alpha(y)}{\alpha(z)})+\alpha(x)\beta(y),\tfrac{\alpha(y)}{\alpha(z)}\big)=\bar\theta(xyz^{-1}x^{-1})\in U_0=\big(-1,1\big)\times\big(\tfrac12,\tfrac32\big).$$
Then $$\alpha(x)\beta(y)\in (-1,1)+[0,1]\cdot (-\tfrac12,\tfrac12)=
    (-\tfrac32,\tfrac32)$$and hence
     $$2^{-n}\le \beta(y)<\frac3{2\alpha(x)}=\frac3{2^{\gamma(x)+1}}<2^{1-\gamma(x)},$$which implies $\gamma(x)<n+1=\varphi_U(x)+1<\psi(x)+1$.
    But this contradicts the choice of the function $\gamma$.
 This contradiction shows that $\varphi_U\ge\psi$ and the monotone map $\varphi_*:\Tau_e(Y_z)\to B_\w(X)$ is cofinal, yielding the desired reduction $\Tau_e(Y_z)\succcurlyeq B_\w(X)$.
\smallskip

5. Let $\mathcal D$ be a discrete family of non-empty open sets in $X$ and $z\in X$ be a non-isolated point with $\kappa=\add(\Tau_z(X))=\add(X)>\w$. Assuming that the uniformity of $X$ is $|\mathcal  D|\cdot\add(X)$-universal, we shall  prove that  $\Tau_e(Y_z)\succcurlyeq\kappa^{\mathcal D}$, where $Y_z=\{xyz^{-1}x^{-1}:x,y\in X\}\subset\FG_u(X)$.
Here we identify $X$ with its image $\delta_X(X)$ in $\FG_u(X)$. The inequality $\add(X)>\w$ implies that $X$ is a uniform $P_u$-space. Then the uniformity $\U_X$ has a base consisting of equivalence relations on $X$.

Indeed, for any entourage $U\in\U_X$ we can choose a sequence $\{U_n\}_{n\in\w}\subset \U_X$ such that $U_0\subset U$, $U_n^{-1}=U_n$ and $U_{n+1}^2\subset U_n$ for all $n\in\w$. Since $X$ is a uniform $P_u$-space, the entourage $E=\bigcap_{n\in\w}U_n$ belongs to $\U_X$. It is clear that $E\subset U$ and $E$ is an equivalence relation.

  Since $z$ is a non-isolated point of $X$ with $\add(\Tau_z(X))=\kappa>\w$, there exists a decreasing sequence $\{V_\alpha\}_{\alpha\in\kappa}\subset\U_X$ of entourages such that $\bigcap_{\alpha\in\kappa}V_\alpha[z]$ is not a neighborhood of $z$. We can additionally assume that each entourage $V_\alpha$, $\alpha\in\kappa$, is an equivalence relation on $X$.

Fix a function $s:\mathcal D\to X$ such that $s(D)\in D$ for all $D\in\mathcal D$ and put $S=s(\mathcal D)$. It is clear that $\kappa^S\cong\kappa^{\mathcal D}$.
Consider the monotone map  $\varphi_*:\Tau_e(Y_z)\to \kappa^S$, assigning to each neighborhood $U\in\Tau_e(Y_z)$ of $e$ the function $$\varphi_U:S\to\kappa,\;\;\varphi_U:x\mapsto
\min\{\alpha\in\kappa: \exists y\in X\setminus V_\alpha\mbox{ such that }xyz^{-1}x^{-1}\in U\}.$$
By analogy with the preceding case, we can show that the function $\varphi_U:S\to\kappa$ is well-defined.

It remains to prove that the monotone map $\varphi_*:\Tau_e(Y_z)\to \kappa^S$, $\varphi_*:U\mapsto\varphi_U$, is cofinal. Given any function $\varphi\in \kappa^S$, we need to find a neighborhood $U\in\Tau_e(Y_z)$ such that $\varphi_U\ge \varphi$.

Choose any ordered field $F$ containing a strictly decreasing transfinite sequence $\{\e(\alpha)\}_{\alpha\in\kappa}\subset(0,1]\subset F$ that tends to zero in $F$, endowed with the topology generated by the linear order. For the existence of such ordered field $F$, see e.g. \cite{LPT}. Replacing $F$ by the subfield generated by the set $\{\e(\alpha)\}_{\alpha\in\kappa}$, we can additionally assume that $|F|=\kappa$. Let $F_+=\{x\in F:x>0\}$ be the multiplicative group consisting of strictly positive elements of the field $F$. Consider the product $F\times F_+$ endowed with the continuous group operation $(x,a)*(y,b)=(x+ay,ab)$. The obtained topological group will be denoted by $F\rtimes F_+$.

Consider the map $b:X\to F$ defined by
$$b(x)=\begin{cases}0,&\mbox{if $x\in\bigcap_{\alpha<\kappa}V_\alpha[z]$},\\
\e\big(\!\min\{\alpha\in\kappa:x\notin V_\alpha[z]\}\big),&\mbox{otherwise}.
\end{cases}
$$
Taking into account that the sequence $\big(\e(\alpha)\big)_{\alpha\in\kappa}$ tends to zero in $F$ and for every $\alpha<\kappa$ the entourages $V_\alpha$ are equivalence relations, we can show that the map $b$ is $\kappa$-continuous.

For every $D\in\mathcal D$ find an equivalence relation $E_D\in\U_X$ such that $E_D[s(D)]\subset D$.
Extend the  function $\varphi:S\to \kappa$ to a $|\mathcal D|$-continuous  function $\bar\varphi:X\to\kappa$ defined by
$$\bar\varphi(x)=\begin{cases}
\varphi(s(D)),&\mbox{if $x\in E_D[s(D)]$ for some $D\in\mathcal D$},\\
0,&\mbox{otherwise}.
\end{cases}
$$
Here we endow the cardinal $\kappa$ with the discrete topology.

Consider the map $a:X\to F_+$, $a(x)=\frac3{2\e(\bar\varphi(x))}$, and observe that it is $|\mathcal D|$-continuous (being locally constant).

 Since the uniformity of $X$ is $|\mathcal D|\cdot\kappa$-universal, the continuous map $\theta:X\to F\rtimes F_+$, $\theta:x\mapsto(b(x),a(x))$, is uniformly continuous and hence admits a unique extension to a continuous group homomorphism $\bar\theta:\FG_u(X)\to F\rtimes F_+$. Consider the open neighborhood $$U_0=\big\{(x,y)\in F\times F_+:-1<x<1,\;\tfrac12<y<\tfrac32\big\}\subset F\rtimes F_+$$ of the neutral element $(0,1)$ of the group $F\rtimes F_+$ and the preimage $U=\bar\theta^{-1}(U_0)\subset\FG_u(X)$ of $U_0$ under the homomorphism $\bar\theta$.
We claim that $\varphi_U\ge \varphi$. To derive a contradiction, assume that $\varphi_U(x)<\varphi(x)$ for some $x\in S$. Then for the ordinal $\alpha=\varphi_U(x)$ there exists a point $y\in \bigcap_{\beta<\alpha}V_\beta\setminus V_\alpha$ such that $xyz^{-1}x^{-1}\in U$. It follows that $b(y)=\e(\alpha)$.

In the group $F\rtimes F_+$ consider the element
$$
\begin{aligned}
\bar\theta(xyz^{-1}x^{-1})&=(b(x),a(x))*(b(y),a(y))*(b(z),a(z))^{-1}*(b(x),a(x))^{-1}=\\
&=\big(b(x)(1-\tfrac{a(y)}{a(z)})+a(x)b(y),\tfrac{a(y)}{a(z)}\big)\in U_0=\big(-1,1\big)\times\big(\tfrac12,\tfrac32\big)
\end{aligned}
$$and observe that
$a(x)b(y) \in (-1,1)+[0,1]\cdot (-\tfrac12,\tfrac12)=
    (-\tfrac32,\tfrac32)\subset F$ and hence
     $\frac3{2\e(\varphi(x))}=a(x)<\frac3{2b(y)}=\frac3{2\e(\alpha)}$ which implies $\e(\varphi(x))>\e(\alpha)$ and $\varphi(x)<\alpha=\varphi_U(x)<\varphi(x)$. This  contradiction shows that $\varphi_U\ge\varphi$ and the monotone map $\varphi_*:\Tau_e(Y)\to \kappa^S$ is cofinal, yielding the desired reduction $\Tau_e(Y)\succcurlyeq \kappa^S\cong\kappa^{\mathcal D}$.
\end{proof}

\begin{theorem}\label{t:FGu2} Let $X$ be a uniform space.
\begin{enumerate}
\item The free topological group $\FG_u(X)$ has an $\w^\w$-base if the uniform space $X$ has an $\w^\w$-base and either $X$ is separable or $\cov^\sharp(X)\le\add(X)$.
\item If the subspace $XX^{-1}=\{xy^{-1}:x,y\in\delta_X(X)\}$ of $\FG_u(X)$ has a neighborhood  $\w^\w$-base at the unit $e$ of $\FG_u(X)$, then the uniform space $\U_X$ has an $\w^\w$-base.
\item If the uniform space $X$ is $\w^\w$-based and $\w$-universal, and $X$ contains a non-$P$-point $z\in X$ such that the subspace $Y_z=\{xy z_\delta^{-1}x^{-1}:x,y\in\delta_X(X)\}\subset\FG_u(X)$ has a neighborhood $\w^\w$-base at $e$, then the space $X$ is separable.
\item If the uniform space $X$ is $\w^\w$-based and $\mathfrak e(\add(X))$-universal, and $X$ contains a non-isolated point $z$ such that $\add(\Tau_z(X))=\add(X)$ and the subspace $Y_z=\{xy z_\delta^{-1}x^{-1}:x,y\in\delta_X(X)\}\subset\FG_u(X)$ has a neighborhood $\w^\w$-base at $e$, then $\cov^\sharp(X)\le\mathfrak e(\add(X))\le\mathfrak d$.
\end{enumerate}
\end{theorem}

\begin{proof} 1. The first statement follows from Corollary~\ref{c:strong-sep} and Theorem~\ref{t:sG-ca}.
\smallskip

2. The second statement follows from Theorem~\ref{t:FGu}(1).
\smallskip

3. Assume that the uniform space $X$ is $\w^\w$-based and $\w$-universal, and $X$ contains a non-$P$-point $z\in X$ such that the subspace $Y_z=\{xy z_\delta^{-1}x^{-1}:x,y\in\delta_X(X)\}\subset\FG_u(X)$ has a neighborhood $\w^\w$-base at $e$. By Theorem~\ref{t:FGu}(4), $\w^\w\succcurlyeq\Tau_e(Y_z)\succcurlyeq B_\w(X)$, which implies that the set $C_\w(X)$ is $\w^\w$-dominated in $\IR^X$. The uniform space $X$, being $\w$-universal, is $\IR$-universal. Then $C_\w(X)=C_u(X)$ and hence the set $C_u(X)$ is $\w^\w$-dominated in $\IR^X$. By Theorem~\ref{t:looong}, the space $X$ is separable.
\smallskip

4. Assume that the uniform space $X$ is $\w^\w$-based and $\mathfrak e(\add(X))$-universal, and $X$ contains a non-isolated point $z$ such that $\add(\Tau_z(X))=\add(X)$ and the subspace $Y_z=\{xyz_\delta^{-1}x^{-1}:x,y\in\delta_X(X)\}\subset\FG_u(X)$ has a neighborhood $\w^\w$-base at $e$. By \cite[2.3.2]{Ban}, $\U_X\succcurlyeq\add(\U_X)$ and hence $\w^\w\succcurlyeq\add(\U_X)$. By Proposition~\ref{p:e(X)}(3), $\add(X)\le\mathfrak e(\add(X))\le \mathfrak d$.

If $\add(\Tau_z(X))=\add(X)=\w$, then $z$ is not a $P$-point and $\mathfrak e(\add(X))=\w_1$ according to Proposition~\ref{p:e(X)}(2). The $\mathfrak e(\add(X))$-universality of the uniform space $X$ implies its $\w$-universality. By the preceding statement, the space $X$ is separable and hence $\cov^\sharp(X)\le\w_1=\mathfrak e(\add(X))\le\mathfrak d$ and we are done.

Now assume that $\add(X)>\w$. In this case $X$ is a $P_u$-space.
To derive a contradiction, assume that $\cov^\sharp(X)>\mathfrak e(\add(X))$. Then there exists an entourage $U\in\U_X$ such that $X\ne U[C]$ for any subset $C\subset X$ of cardinality $|C|<\mathfrak e(\add(X))$. Choose an entourage $V\in\U_X$ such that $V^{-1} V\subset U$. Using Zorn Lemma, choose a maximal subset $C\subset X$ such that $V[x]\cap V[y]=\emptyset$ for any distinct points $x,y\in C$. By the maximality of $C$, for every $x\in X$ there is a point $c\in C$ such that $V[x]\cap V[c]\ne\emptyset$, which implies $x\in V^{-1}V[c]\subset U[C]$. So, $X=U[C]$ and $|C|\ge\mathfrak e(\add(X))$ by the choice of $U$.  
Choose an entourage $W\in\U_X$ such that $W^{-1}W\subset V$ and observe that $\mathcal D=\{W[c]\}_{c\in C}$ is a discrete family of $W$-balls in $X$ with $|\mathcal D|=|C|\ge\mathfrak e(\add(X))$. Theorem~\ref{t:FGu}(5) implies that $\w^\w\succcurlyeq  \Tau_e(Y_z)\succcurlyeq\add(X)^{\mathcal D}\succcurlyeq\add(X)^{\mathfrak e(\add(X))}$, which contradicts the definition of $\mathfrak e(\add(X))$. This contradiction completes the proof of the inequality $\cov^\sharp(X)\le\mathfrak e(\add(X))$. \end{proof}

For Tychonoff spaces Theorem~\ref{t:FGu2} implies:

\begin{theorem}\label{t:FG} Let $X$ be a Tychonoff space.
\begin{enumerate}
\item The free topological group $\FG(X)$ has an $\w^\w$-base if  the universal uniformity $\U_X$ of $X$ has an $\w^\w$-base and either $X$ is separable or $\cov^\sharp(X)\le\add(X)$.
\item If the subspace $XX^{-1}=\{xy^{-1}:x,y\in\delta_X(X)\}$ of $\FG(X)$ has a neighborhood  $\w^\w$-base at the unit $e$ of $\FG(X)$, then the universal uniformity $\U_X$ of $X$ has an $\w^\w$-base.
\item If $X$ is universally $\w^\w$-based and $X$ contains a non-$P$-point $z\in X$ such that the subspace $Y_z=\{xy z_\delta^{-1}x^{-1}:x,y\in\delta_X(X)\}\subset\FG_u(X)$ has a neighborhood $\w^\w$-base at $e$, then the space $X$ is separable.
\item If $X$ is not discrete  and the subspace $Y=\{xyz^{-1}x^{-1}:x,y,z\in\delta_X(X)\}$ of $\FG(X)$ has a neighborhood $\w^\w$-base at  $e$, then the universal uniformity $\U_X$ of $X$ has an $\w^\w$-base and $\cov^\sharp(X)\le\mathfrak e(\add(X))\le\mathfrak d$.
\end{enumerate}
\end{theorem}

\begin{proof} Observe that all statements of this theorem trivially hold if the space $X$ is discrete. So, we assume that $X$ is not discrete. Endowing the space $X$ with the universal uniformity, we observe that $\FG(X)=\FG_u(X)$.
\smallskip

1--3. The first three statements follow from the corresponding statements of Theorem~\ref{t:FGu2}.
\smallskip

4. To prove the fourth statement, assume that $X$ is not discrete and the subspace $Y=\{xyz^{-1}x^{-1}:x,y,z\in\delta_X(X)\}$ of $\FG(X)$ has a local $\w^\w$-base at  $e$. Since $XX^{-1}\subset Y$, we can apply the second statement and conclude that the universal uniformity of $X$ has an $\w^\w$-base. By \cite[7.8.14]{Ban}, there exists a point $z\in X$ such that $\add(\Tau_z(X))=\add(X)$. By Theorem~\ref{t:FGu2}(4), $\cov^\sharp(X)\le\mathfrak e(\add(X))\le\mathfrak d$.
\end{proof}

Now we prove several characterizations of Tychonoff spaces whose free topological groups have $\w^\w$-bases. The following theorem is an immediate consequence of Theorem~\ref{t:FG}.

\begin{theorem}\label{t:nP-FG} Let $X$ be a Tychonoff space. If  $X$ is not a $P$-space, then the free topological group $\FG(X)$ has an $\w^\w$-base iff $X$ is a separable universally $\w^\w$-based space.
\end{theorem}

Observe that Theorem~\ref{t:FG} implies the following (consistent) characterization.

\begin{theorem}\label{t:P-FG} If $X$ is a non-discrete Tychonoff space, then we have the implications $(1)\Ra(2)\Ra(3)$ between the following statements:
\begin{enumerate}
\item the universal uniformity $\U_X$ has an $\w^\w$-base and $X$ is either separable or  $\cov^\sharp(X)\le\add(X)$;
\item the free topological group $\FG(X)$ has an $\w^\w$-base;
\item the universal uniformity $\U_X$ has an $\w^\w$-base and $X$ is either separable or  $\cov^\sharp(X)\le\mathfrak e(\add(X))$.
\end{enumerate}
If $\mathfrak e(\add(X))\le\add(X)$, then the conditions $(1)$--$(3)$ are equivalent.
\end{theorem}

Our final characterization holds under the set-theoretic assumption $\mathfrak e^\sharp=\w_1$ which is strictly weaker than $\mathfrak b=\mathfrak d$ (see \cite[2.4.9]{Ban}).

\begin{corollary}\label{c:Cons}  Assume that $\mathfrak e^{\sharp}=\w_1$.
For a Tychonoff space $X$ the following conditions are equivalent:
\begin{enumerate}
\item the free topological group $\FG(X)$ has an $\w^\w$-base;
\item the subspace $Y=\{xyz^{-1}x^{-1}:x,y,z\in\delta_X(X)\}$ of $\FG(X)$ has a neighborhood $\w^\w$-base at the unit $e$ of $\FG(X)$;
\item the space $X$ is universally $\w^\w$-based and $X$ is either separable or $\cov^\sharp(X)\le \add(X)$.
\end{enumerate}
\end{corollary}

\begin{proof} The implication $(3)\Ra(1)$ was proved in Theorem~\ref{t:FG}(1) and $(1)\Ra(2)$ is trivial.

To prove that $(2)\Ra(3)$, assume that the subspace $Y$ of $\FG(X)$ has a neighborhood $\w^\w$-base at the unit $e$ of $\FG(X)$. By Theorem~\ref{t:FG}(2), the space $X$ is universally $\w^\w$-based. It remains to prove that $X$ is separable or $\cov^\sharp_\w(X)\le\add(X)$. Assume that $X$ is not separable. If the space $X$ is discrete, then $\cov^\sharp(X)<\infty=\add(X)$ and we are done.
So, we assume that $X$ is not discrete and not separable. Theorem~\ref{t:nP-FG} implies that $X$ is a $P$-space and hence $\add(X)>\w$. By Theorem~\ref{t:FG}(4), $\cov^\sharp(X)\le\mathfrak e(\add(X))$.
 By Proposition~\ref{p:e=w1}, the equality $\mathfrak e^\sharp=\w_1$ implies $\mathfrak e(\add(X))\le\add(X)$ and hence $\cov_\w^\sharp(X)\le \mathfrak e(\add(X))\le\add(X)$.
\end{proof}

The implication $(2)\Ra(3)$ in Corollary~\ref{c:Cons} does not hold in ZFC.

\begin{example}\label{e:FG:b<n} Assume that $\w^\w\succcurlyeq \kappa^\delta$ for some uncountable regular cardinals $\delta\ge\kappa$. Consider the topological sum $X=\delta\sqcup 2^{\odot\kappa}$ of the cardinal $\delta$ endowed with the discrete topology and the space
$$2^{\odot\kappa}=\{x\in 2^{\kappa}:|\{\alpha\in\kappa:x(\alpha)\ne 0\}|<\w\}$$ endowed with the topology generated by the base consisting of the sets $U_\alpha[x]=\{y\in 2^{\odot\kappa}:y|\alpha=x|\alpha\}$ where $x\in 2^{\odot\kappa}$ and $\alpha\in\kappa$. The space $X=\delta\sqcup 2^{\odot\kappa}$ has the following properties:
\begin{enumerate}
\item $|X|=\delta$, $\cov^\sharp(X)=\delta^+>\kappa=\add(X)=\add(\U_X)=\cof(\U_X)\cong\U_X$;
\item the free topological group $\FG(X)$ has an $\w^\w$-base.
\end{enumerate}
\end{example}

\begin{proof} By \cite[8.1.1(2)]{Ban}, the universal uniformity of the space $2^{\odot\kappa}$ is generated by the base consisting of the entourages $U_\alpha=\{(x,y)\in 2^{\odot\kappa}\times 2^{\odot\kappa}:x|\alpha=y|\alpha\}$ where $\alpha\in\kappa$. This description implies the first statement. By Theorem~\ref{t:FGu}(2), $\w^\w\succcurlyeq\kappa^\delta\cong \U_X^{|X|}\succcurlyeq \FG(X)$, which means that $\FG(X)$ has an $\w^\w$-base.
\end{proof}

\begin{corollary} If $\mathfrak e^\sharp>\w_1$, then there exists a non-separable universally $\w^\w$-based Tychonoff space $X$ such that $\w<\add(X)<\cov^\sharp(X)$ and the free topological group $\FG(X)$ has an $\w^\w$-base.
\end{corollary}

We finish this section by some consistent examples.

\begin{example}\label{ex:w1=b} Under $\w_1=\mathfrak b$ there exists a non-separable Lindel\"of Tychonoff space $X$ whose free topological group $\FG(X)$ has an $\w^\w$-base.
\end{example}

\begin{proof} Consider the space $$2^{\odot\w_1}=\{x\in 2^{\w_1}:|\{\alpha\in\w_1:x(\alpha)\ne 0\}|<\w\}$$ endowed with the topology generated by the base consisting of the sets $U_\alpha[x]=\{y\in 2^{\odot\w_1}:y|\alpha=x|\alpha\}$ where $x\in 2^{\odot{\w_1}}$ and $\alpha\in\w_1$. By \cite[8.1.1]{Ban},
the space $X$ is Lindel\"of, non-separable, and $\cov^\sharp(X)=\w_1=\add(X)\cong \U_X$. Moreover, under $\w_1=\mathfrak b$ the space $X$ is universally $\w^\w$-based. By Theorem~\ref{t:FG}(1), the free topological group $\FG(X)$ has an $\w^\w$-base.
\end{proof}

\begin{example} Under $\w_1=\mathfrak b$ there exists a Lindel\"of Tychonoff space $X$ such that the universal uniformity $\U_X$ of $X$ has $\w^\w$-base but the free topological group $\FG(X)$ fails to have an $\w^\w$-base.
\end{example}

\begin{proof} Let $X=K\sqcup 2^{\odot \w_1}$ be the topological sum of any infinite compact metrizable space $K$ and the non-separable Lindel\"of space $2^{\odot\w_1}$ from Example~\ref{ex:w1=b}. By \cite[8.1.1]{Ban}, the space $2^{\odot\w_1}$ is universally $\w^\w$-based and so is the space $X$. Observe that $X$ is not a $P$-space. Assuming that the free topological group $\FG(X)$ has an $\w^\w$-base, we can apply Theorem~\ref{t:FG}(3) and conclude that the space $X$ is separable, which is not true. This contradiction shows that the free topological group $\FG(X)$ fails to have an $\w^\w$-base.
\end{proof}

The results of this section suggest to ask the following (open) problem.

\begin{problem}\label{p:open} Find a ZFC-characterization of Tychonoff spaces whose free topological group $\FG(X)$ has an $\w^\w$-base.
\end{problem}

%Observe that Question~\ref{q:open} has affirmative answer under the assumption $\mathfrak e^{\sharp}=\w_1$, so the task is show that this question has (or does not have) an affirmative answer in %ZFC.

\section{Proof of Theorem~\ref{t:mainU}}\label{s:pf1}

Given a uniform space $X$ consider the conditions:
\begin{enumerate}
\item[$(\AG)$] the free Abelian topological group $\AG_u(X)$ of $X$ has an $\w^\w$-base;
\item[$(\BG)$] the free Boolean topological group $\BG_u(X)$ of $X$ has an $\w^\w$-base;
\item[$(\FG)$] the free topological group $\FG_u(X)$ of $X$ has an $\w^\w$-base;
\item[$(\Lc)$] the free locally convex space $\Lc_u(X)$ of $X$ has an $\w^\w$-base;
\item[$(\Lin)$] the free topological vector space $\Lin_u(X)$ of $X$ has an $\w^\w$-base;
\item[$(\U)$] the uniformity $\U_X$ of $X$ has an $\w^\w$-base;
\item[$(C)$] the poset $C_u(X)$ is $\w^\w$-dominated;
\item[$(\sigma)$] the space $X$ is $\sigma$-compact;
\item[$(s)$] the space $X$ is separable;
\item[$(S)$] $X$ is separable or $\cov^\sharp(X)\le\add(X)$.
\end{enumerate}
We need to prove the following statements:
\begin{enumerate}
\item[\textup{(1)}]
$(\U{\wedge}C)\Leftrightarrow(\Lc)\Leftarrow (\Lin)\Leftarrow(\U{\wedge}s)\Leftarrow(\U{\wedge}\sigma)\Ra(\U{\wedge}S)\Ra(\FG)\Ra(\AG)\Leftrightarrow(\BG)\Leftrightarrow(\U)$ and moreover $(\Lin)\Leftrightarrow(\U{\wedge}s)$ under $\w_1<\mathfrak b$.
\item[\textup{(2)}]
If the completion of the uniform space $X$ is $\IR$-complete, then $(\U\wedge C)\Leftrightarrow(\Lc)\Leftrightarrow(\Lin)\Leftrightarrow(\U{\wedge}s)$.
\item[\textup{(3)}]
If the uniform space $X$ is $\IR$-universal, then $(\U\wedge C)\Leftrightarrow(\Lc)\Leftrightarrow(\Lin)\Leftrightarrow(\U{\wedge}s)\Leftrightarrow(\U{\wedge}\sigma)$.
\item[\textup{(4)}]
If the space $X$ is separable, then $(\Lc)\Leftrightarrow (\Lin)\Leftrightarrow(\FG)\Leftrightarrow(\AG)\Leftrightarrow(\BG)\Leftrightarrow(\U)$.
\item[\textup{(5)}]
If the uniform space $X$ is $\w$-universal and $X$ is not a $P$-space, then\newline
$(\U{\wedge}C)\Leftrightarrow(\Lc)\Leftrightarrow (\Lin)\Leftrightarrow(\U{\wedge}s)\Leftrightarrow(\U{\wedge}\sigma)\Leftrightarrow(\FG)
$.
\end{enumerate}
\smallskip

1. The implications $(\U{\wedge}\sigma)\Ra(\U{\wedge}s)\Ra(\Lin)\Ra(\Lc)\Leftrightarrow(\U{\wedge}C)$ were proved in Theorem~\ref{t:Lin+LcU}, $(\U{\wedge}s)\Ra(\U{\wedge}S)$ is trivial, $(\U{\wedge}S)\Ra(\FG)$ was proved in Theorem~\ref{t:FG}(1), $(\FG)\Ra(\AG)$ follows from the openness of the canonical homomorphism $\FG_u(X)\to\AG_u(X)$, and $(\AG)\Leftrightarrow(\BG)\Leftrightarrow(\U)$ were proved in Theorem~\ref{t:ABGu}. Under $\w_1<\mathfrak b$ the equivalence $(\U{\wedge}s)\Leftrightarrow(\Lin)$ was proved in Theorem~\ref{t:Lin+LcU}.
\smallskip

2. If the completion of the uniform space $X$ is $\IR$-complete, then the equivalence of the conditions $(\U{\wedge}C)$ and $(\U{\wedge}s)$ follows from Theorem~\ref{t:looong}. Combining the equivalence
$(\U{\wedge}C)\Leftrightarrow(\U{\wedge}s)$ with the implications $(\U{\wedge}s)\Ra(\Lin)\Ra(\Lc)\Ra(\U{\wedge}C)$ proved in the first statement, we get the equivalences
$(\U{\wedge}s)\Leftrightarrow(\Lin)\Leftrightarrow(\Lc)\Leftrightarrow(\U{\wedge}C)$.
\smallskip

3. If the uniform space $X$ is $\IR$-universal, then by Theorem~\ref{t:looong}, the conditions $(\U{\wedge}C)$ and $(\U{\wedge}\sigma)$ are equivalent. Combined this equivalence with the implications proved in the first statement, we obtain the implications
$$(\U{\wedge}C)\Leftrightarrow(\Lin)\Leftrightarrow(\Lc)\Leftrightarrow
(\U{\wedge}s)\Leftrightarrow(\U{\wedge}\sigma)\Ra(\U{\wedge}S)\Ra(\FG)\Ra(\AG)\Leftrightarrow(\BG)\Leftrightarrow(\U).
$$

4. If the uniform space $X$ is separable, then the chain of implications from the first statement turns into the chain of equivalences $(\U)\Leftarrow(\Lc)\Leftarrow(\Lin)\Leftarrow(\U)\Ra(\FG)\Ra(\AG)\Ra(\BG)\Ra(\U)$.
\smallskip

5. Finally, assume that the uniform space $X$ is $\w$-universal and the topological space $X$ is not a $P$-space. Consequently, $X$ contains a non-isolated point $z\in X$ such that $\add(\Tau_z(X))=\w=\add(X)$. Theorem~\ref{t:FGu2}(3) yields the implication $(\FG)\Ra(\U{\wedge}s)$.  The $\w$-universality of $X$ implies the $\IR$-universality of $X$. Combining the implication $(\FG)\Ra(\U{\wedge}s)$ with the implications established in the statement (3), we get the equivalences $$(\U{\wedge}C)\Leftrightarrow(\Lc)\Leftrightarrow (\Lin)\Leftrightarrow(\U{\wedge}s)\Leftrightarrow(\U{\wedge}\sigma)\Leftrightarrow(\FG).$$

\section{Proof of Theorem \ref{t:mainT}}\label{s:pf2}

Given a Tychonoff space $X$, consider the following conditions:
\begin{enumerate}
\item[$(\AG)$] the free Abelian topological group $\AG(X)$ of $X$ has an $\w^\w$-base;
\item[$(\BG)$] the free Boolean topological group $\BG(X)$ of $X$ has an $\w^\w$-base;
\item[$(\FG)$] the free topological group $\FG(X)$ of $X$ has an $\w^\w$-base;
\item[$(\Lc)$] the free locally convex space $\Lc(X)$ of $X$ has an $\w^\w$-base;
\item[$(\Lin)$] the free topological vector space $\Lin(X)$ of $X$ has an $\w^\w$-base;
\item[$(\U)$] the universal uniformity $\U_X$ of $X$ has an $\w^\w$-base;
\item[$(C)$] the poset $C(X)$ is $\w^\w$-dominated;
\item[$(s)$] the space $X$ is separable;
\item[$(S)$] $X$ is separable or $\cov^\sharp(X)\le\add(X)$.
\item[$(D)$] $X$ is discrete;
\item[$(\sigma)$] $X$ is $\sigma$-compact;
\item[$(\sigma')$] the set $X'$ of non-isolated points of $X$ is $\sigma$-compact.
\end{enumerate}

We need to prove the following statements:
\begin{enumerate}
\item[\textup{(1)}]
$(\U{\wedge}C)\Leftrightarrow(\Lc)\Leftrightarrow (\Lin)\Leftrightarrow(\U{\wedge}s)\Leftrightarrow(\U{\wedge}\sigma)\Ra(\U{\wedge}S)\Ra(\FG)\Ra(\AG)\Leftrightarrow(\BG)\Leftrightarrow(\U)$ and moreover  $(\U{\wedge}S)\Leftrightarrow (\FG)$ under the set-theoretic assumption $\mathfrak e^\sharp=\w_1$ (which is weaker than $\mathfrak b=\mathfrak d$).
\item[\textup{(2)}]
If $X$ is not a $P$-space, then $(\U{\wedge}s)\Leftrightarrow(\FG)$.
\item[\textup{(3)}]
If the space $X$ is separable, then $(\Lc)\Leftrightarrow (\Lin)\Leftrightarrow(\FG)\Leftrightarrow(\AG)\Leftrightarrow(\BG)\Leftrightarrow(\U)$.
\item[\textup{(4)}]
If the space $X$ is metrizable, then
$(\Lc)\Leftrightarrow (\Lin)\Leftrightarrow(\sigma)\Ra(D{\vee}\sigma)\Leftrightarrow(\FG)\Ra(\AG)\Leftrightarrow(\BG)\Leftrightarrow(\sigma')$. \end{enumerate}

1. Considering $X$ as a uniform space endowed with the universal uniformity $\U_X$ and applying Theorem~\ref{t:mainU}(1,3), we obtain the implications:
$$(\U{\wedge}C)\Leftrightarrow(\Lc)\Leftrightarrow (\Lin)\Leftrightarrow(\U{\wedge}s)\Leftrightarrow(\U{\wedge}\sigma)\Ra(\U{\wedge}S)\Ra(\FG)\Ra(\AG)\Leftrightarrow(\BG)\Leftrightarrow(\U).$$
If $\mathfrak e^\sharp=\w_1$, then $(\U{\wedge}S)\Leftrightarrow(\FG)$ by Corollary~\ref{c:Cons}.
\smallskip

2. If $X$ is not a $P$-space, then the equivalence $(\U{\wedge}s)\Leftrightarrow(\FG)$ follows from Theorem~\ref{t:nP-FG}.
\smallskip

3. If the space $X$ is separable, then the equivalences $(\Lc)\Leftrightarrow (\Lin)\Leftrightarrow(\FG)\Leftrightarrow(\AG)\Leftrightarrow(\BG)\Leftrightarrow(\U)$  follow from Theorem~\ref{t:mainU}(4).

4. Assuming that the space $X$ is metrizable, we should prove the equivalences
$$(\Lc)\Leftrightarrow(\Lin)\Leftrightarrow(\U{\wedge}s)\Leftrightarrow(\sigma)\Ra(D{\vee}\sigma)\Leftrightarrow
(\U{\wedge}S)\Leftrightarrow(\FG)\Ra(\AG)\Leftrightarrow(\BG)\Leftrightarrow(\U)\Leftrightarrow(\sigma').$$
Taking into account the implications proved in the first statement, it suffices to show that
$(\U)\Leftrightarrow(\sigma')$ and $(D{\vee}\sigma)\Ra(\U{\wedge}S)\Ra(\FG)\Ra(D{\vee}\sigma)$. The implications
$(\U)\Leftrightarrow(\sigma')$ and $(D{\vee}\sigma)\Ra(\U{\wedge}S)$ follow from Theorems~\ref{t:uwwb}(6) and \ref{t:Lasnev}; the implication $(\U{\wedge}S)\Ra(\FG)$ was proved in Theorem~\ref{t:FG}(1). To prove that $(\FG)\Ra(D{\vee}\sigma)$, assume that the space $X$ is not discrete and the free topological group $\FG(X)$ has an $\w^\w$-base. Since the non-discrete metrizable space $X$ is not a $P$-space, we can apply Theorem~\ref{t:FG}(2,3) and conclude that the space $X$ is separable and universally $\w^\w$-based. By Theorem~\ref{t:looong}, the space $X$ is $\sigma$-compact.

\section{Topological spaces $X$ with the inductive topology}
We say that a topological space $X$  has the {\em inductive topology} with respect to a cover $\C$ of $X$ if a subset $V\subseteq X$ is open in $X$ if and only if for every $C\in\C$ the intersection $V\cap C$ is relatively open in $C$.

\begin{theorem}\label{ind_limit} Assume that a Tychonoff space $X$ has the inductive topology with respect to a countable cover $\{X_n\}_{n\in\omega}$ of $X$. Let $\FO\in\{\AG,\BG,\Lc,\Lin\}$. The free object $\FO(X)$ has an $\w^\w$-base if for every $n\in\w$ the space $\FO(X_n)$ has an $\w^\w$-base.
\end{theorem}

\begin{proof} Assume that for every $n\in\w$ the space $\FO(X_n)$ has an $\w^\w$-base. By Theorems~\ref{t:mainT}, for every $n\in\w$ the universal uniformity $\U_{X_n}$ of the space $X_n$ has an $\w^\w$-base. Let $Y=\bigcup_{n\in\w}\{n\}\times X_n\subset \w\times X$ be the topological sum of the family $(X_n)_{n\in\w}$ and $q:Y\to X$, $q:(n,x)\mapsto x$, be the natural projection.
Since the space $X$ carries the inductive topology with respect to the cover $\{X_n\}_{n\in\w}$, the projection $q:Y\to X$ is a $\IR$-quotient map.

Since $$(\w^\w)^\w\succcurlyeq \prod_{n\in\w}\U_{X_n}\cong\U_Y,$$ the universal uniformity $\U_Y$ of the space $Y$ has an $\w^\w$-base. Taking into account that the space $X$ is the image of the universally $\w^\w$-based space $Y$ under the $\IR$-quotient map $q$, we conclude that the space $X$ is universally $\w^\w$-based (see Proposition~\ref{p:Rquot}). If $\FO\in\{\AG,\BG\}$, then the free object $\FO(X)$ has an $\w^\w$-base by Corollary~\ref{c:ABG}.

If $\FO\in\{\Lc,\Lin\}$, then by Theorem~\ref{t:Lin+LcT}, the spaces $X_n$, $n\in\w$, are separable and so are the spaces $Y$ and $X$. By Theorem~\ref{t:Lin+LcT}, the free (locally convex) topological vector space $\FO(X)$ has an $\w^\w$-base.
\end{proof}

Theorem~\ref{ind_limit} is not true for the free object $\FO=\FG$.

\begin{example} Let $X=D\sqcup K$ be the topological sum of an uncountable discrete space $D$ and an infinite compact metrizable space $K$. The free topological groups $\FG(D)$ and $\FG(K)$ have $\w^\w$-bases but $\FG(X)$ does not (see Theorem~\ref{t:nP-FG}).
\end{example}

\begin{problem} Assume that a Tychonoff space $X$ has the inductive topology with respect to a cover $\{X_n\}_{n\in\w}$ of $X$ such that $X_n\subset X_{n+1}$ for all $n\in\w$. Does the free topological group $\FG(X)$ have an $\w^\w$-base if for every $n\in\w$ the free topological group $\FG(X_n)$ has an $\w^\w$-base?
\end{problem}

The following theorem gives a partial affirmative answer to this problem.

\begin{theorem} Assume that a Tychonoff space $X$ has the inductive topology with respect to a countable cover $\{X_n\}_{n\in\omega}$ such that $X_n\subset X_{n+1}$ for all $n\in\w$. Assume that for every $n\in\w$ the free topological group $\FG(X_n)$ has an $\w^\w$-base. If $\mathfrak e^\sharp=\w_1$ or $X$ is not a non-discrete $P$-space, then the free topological group $\FG(X)$ has an $\w^\w$-base.
\end{theorem}

\begin{proof} By analogy with Theorem~\ref{ind_limit}, we can prove that the universal uniformity $\U_X$ of the space $X$ has an $\w^\w$-base. If infinitely many spaces $X_n$ are discrete, then the space $X$ is discrete and so is its free topological group $\FG(X)$. In this case $\FG(X)$ trivially has an $\w^\w$-base.

So, we assume that only finitely many spaces $X_n$ are discrete.  If
infinitely many spaces $X_n$ are separable, then the space $X=\bigcup_{n\in\w}X_n$ is separable, too.
By Theorem~\ref{t:mainT}, the free topological group $\FG(X)$ has an $\w^\w$-base.

So, we assume that only finitely many spaces $X_n$ are separable. Replacing the sequence $(X_n)_{n\in\w}$ by a suitable subsequence, we can assume that all spaces $X_n$ are not discrete and not separable. Since for every $n\in\w$ the free topological group $\FG(X_n)$ has an $\w^\w$-base, we can apply Theorem~\ref{t:nP-FG} and conclude that the non-discrete space $X_n$ is a $P$-space.
Then $X$ is a non-discrete $P$-space, too and by our assumption, $\mathfrak e^\sharp=\w_1$.
By Corollary~\ref{c:Cons}, $\cov^\sharp(X_n)\le\add(X_n)$ for all $n\in\w$. By \cite[7.8.14]{Ban}, we get $\add(X)=\min\{\add(\Tau_x(X)):x\in X'\}$ and $\add(X_n)=\min\{\add(\Tau_x(X_n)):x\in X_n'\}>\w$ for all $n\in\w$. This equalities imply that $(\add(X_n))_{n\in\w}$ is a decreasing sequence of cardinals with $\add(X)=\min\{\add(X_n):n\in\w\}$. So, we can find a number $k\in\w$ such that $\add(X)=\add(X_n)\ge\cov^\sharp(X_n)$ for all $n\ge k$. The regularity of the uncountable cardinal $\add(X)\ge\cov^\sharp(X_n)$, $n\ge k$, implies that $\cov^\sharp(X)\le\add(X)$. By Theorem~\ref{t:FG}(1), the free topological group $\FG(X)$ has an $\w^\w$-base.
\end{proof}

\end{document}